\documentclass[letterpaper, 12pt]{amsart}

\usepackage{latexsym,mathrsfs,amscd,amsmath,verbatim,amssymb,calc,enumerate,amsxtra,
ifpdf}
\usepackage[all]{xy}
\ifpdf%
  \usepackage{hyperref}
\else%
  \usepackage[hypertex]{hyperref}
\fi%

\allowdisplaybreaks[1]

\theoremstyle{plain}
\newtheorem{theorem}{Theorem}[section]
\newtheorem{corollary}[theorem]{Corollary}
\newtheorem{proposition}[theorem]{Proposition}
\newtheorem{lemma}[theorem]{Lemma}
\newtheorem*{theorem*}{Theorem}

\theoremstyle{remark}
\newtheorem{remark}[theorem]{Remark}

\theoremstyle{definition}
\newtheorem{definition}[theorem]{Definition}
\newtheorem{example}[theorem]{Example}

\numberwithin{equation}{theorem}

{\noindent\ignorespaces}%
{\par\noindent%
\ignorespacesafterend}

\newcommand{\ZZ}{\mathbb{Z}}
\newcommand{\m}{\mathfrak{m}}
\newcommand{\n}{\mathfrak{n}}
\newcommand{\p}{\mathfrak{p}}
\newcommand{\q}{\mathfrak{q}}
\newcommand{\Ass}{\operatorname{Ass}}
\newcommand{\Spec}{\operatorname{Spec}}
\newcommand{\Min}{\operatorname{Min}}
\newcommand{\Max}{\operatorname{Max}}
\newcommand{\Supp}{\operatorname{Supp}}
\newcommand{\Ann}{\operatorname{Ann}}
\newcommand{\Var}{\operatorname{Var}}
\newcommand{\Jac}{\operatorname{Jac}}
\newcommand{\im}{\operatorname{im}}
\newcommand{\Hom}{\operatorname{Hom}}
\newcommand{\inj}{\operatorname{inj}}
\newcommand{\spl}{\operatorname{spl}}
\newcommand{\sur}{\operatorname{sur}}
\newcommand{\rank}{\operatorname{rank}}
\newcommand{\Tor}{\operatorname{Tor}}

\newcommand{\coker}{\operatorname{coker}}
\newcommand{\GL}{\operatorname{\textbf{GL}}}
\newcommand{\Pic}{\operatorname{Pic}}

\begin{document}

\title[Surjective and Splitting Capacities]{Surjective and Splitting Capacities}
\author{Robin Baidya}
\email{rbaidya1@gsu.edu}
\address{Department of Mathematics and Statistics, Georgia State University, Atlanta, Georgia 30303}
\curraddr{}
\thanks{This paper constitutes a portion of a dissertation~\cite{Bai3} that will be submitted to the Department of Mathematics and Statistics at Georgia State University in partial fulfillment of the requirements for the degree of Doctor of Philosophy.  The author would like to thank his advisor Yongwei Yao for suggesting the problems considered in this paper, for proposing several definitions that ultimately led to solutions of these problems, and for reviewing many early drafts of this work.  The author would also like to thank Florian Enescu for his helpful comments during a presentation of some of these results at the algebra seminar at Georgia State University.}
\date{\today}

\begin{abstract}
Let $R$ be a commutative ring, $S$ a module-finite $R$-algebra, $M$ a right $S$-module, and $N$ a finitely generated right $S$-module such that $\Max(R)\cap\Supp_R(N)$ is finite-dimensional and Noetherian.  Working under various combinations of additional hypotheses on $R$, $M$, and $N$, we give lower bounds on the \textit{global surjective capacity of $M$ with respect to $N$ over $S$}, that is, the supremum of the nonnegative integers $t$ such that $N^{\oplus t}$ is a direct summand of a quotient $S$-module of $M$.  We express our lower bounds in terms of local analogues of global surjective capacity and topological properties of $\Spec(R)$.  Assuming that $N$ is finitely presented over $S$, we also give lower bounds on the \textit{global splitting capacity of $M$ with respect to $N$ over $S$}, that is, the supremum of the nonnegative integers $t$ such that $N^{\oplus t}$ is a direct summand of $M$.  In the process, we generalize Serre's Splitting Theorem from algebraic $K$-theory and a theorem from stable algebra due to Bass.  We also generalize a theorem on Noetherian modules by De~Stefani, Polstra, and Yao that serves as an analogue of an older result by Stafford.  To close, we consider the case of finitely generated modules over a Dedekind domain; in this case, we show that we can provide conditions equivalent to having a given global surjective or splitting capacity.
\end{abstract}
\commby{}
\maketitle

\setcounter{section}{-1} 

\section{Introduction}\label{sec:intro}

In this paper, every ring has a 1, and every module is unital.

Let $R$ be a commutative ring, and let $M$ and $N$ be $R$-modules.  It is well known that an $R$-linear map $f:M\rightarrow N$ is surjective if and only if $f_{\p}$ is surjective for every $\p\in\Spec(R)$.  Suppose now that, for every $\p\in\Spec(R)$, there is a surjective $R_{\p}$-linear map from $M_{\p}$ to $N_{\p}$.  Can we conclude that there is a surjective $R$-linear map from $M$ to $N$?

In general, the answer is \textit{no}.  For example, suppose that $R$ is a Dedekind domain with a nonprincipal ideal $I$.  Then $R_{\p}\cong I_{\p}$ for every $\p\in\Spec(R)$, but there is neither a surjective $R$-linear map from $R$ to $I$ nor a surjective $R$-linear map from $I$ to $R$.

How might we strengthen our hypothesis on $R$, $M$, and $N$, then, to be able to conclude that there is a surjective $R$-linear map from $M$ to $N$ or, more generally, from $M$ to $N^{\oplus t}$ for a given positive integer $t$?  Before considering this question further, we introduce some definitions and notation:

\begin{definition}\label{definition:sur}
Let $R$ be a commutative ring, $S$ an $R$-algebra, and $M$ and $N$ right $S$-modules.  We let $\sur_S(M,N)$ denote the supremum of the nonnegative integers $t$ such that there exists a surjective $S$-linear map from $M$ to $N^{\oplus t}$, and we refer to $\sur_S(M,N)$ as the \textit{global surjective capacity of $M$ with respect to $N$ over $S$}.

Let $\p\in\Spec(R)$.  We refer to $\sur_{S_{\p}}(M_{\p},N_{\p})$ as the \textit{local surjective capacity of $M$ with respect to $N$ over $S$ at $\p$}.
\end{definition}

Under the hypotheses of the previous definition, we observe that we can always get an upper bound on a global surjective capacity in terms of local surjective capacities:
\[
\sur_S(M,N)\leqslant\inf\{\sur_{S_{\m}}(M_{\m},N_{\m}):\m\in\Max(R)\cap\Supp_R(N)\}.
\]
Can we also get a lower bound on a global surjective capacity using local surjective capacities?  Bass gives an answer to this question in his work on stable algebra~\cite[Theorem~8.2]{Bas}.  In the sequel, we endow every subset $Y$ of $\Spec(R)$ with the subspace topology induced by the Zariski topology on $\Spec(R)$, and we define $\dim(Y)$ to be the Krull dimension of $Y$ as a topological space.  We define $\dim(\varnothing)=-\infty$.

\begin{theorem}[Bass~{\cite[Theorem~8.2]{Bas}}]\label{theorem:Bas}
Let $R$ be a commutative ring, $S$ a module-finite $R$-algebra, and $M$ a direct summand of a direct sum of finitely presented right $S$-modules.  Suppose that $Y:=\Max(R)$ is Noetherian with $\dim(Y)<\infty$.  Suppose also that $\sur_{S_{\m}}(M_{\m},S_{\m})\geqslant 1+\dim(Y)$ for every $\m\in Y$.  Then $\sur_S(M,S)\geqslant 1$.
\end{theorem}

We generalize this theorem in two ways:  We replace the module $S$ with an arbitrary finitely generated right $S$-module $N$, and we change the number 1 to an arbitrary positive integer $t$.  With the help of these modifications, we get two new conclusions.

\begin{theorem}\label{theorem:sur-Bas}
Let $R$ be a commutative ring, $S$ a module-finite $R$-algebra, $M$ a direct summand of a direct sum of finitely presented right $S$-modules, and $N$ a finitely generated right $S$-module.  Suppose that $Y:=\Max(R)\cap\Supp_R(N)$ is Noetherian with $\dim(Y)<\infty$.  Then the following statements hold:
\begin{enumerate}
\item Let $t$ be a positive integer, and suppose that $\sur_{S_{\m}}(M_{\m},N_{\m})\geqslant t+\dim(Y)$ for every $\m\in Y$.  Then $\sur_S(M,N)\geqslant t$.
\item $\sur_S(M,N)=\infty$ if and only if $\sur_{S_{\m}}(M_{\m},N_{\m})=\infty$ for every $\m\in Y$.
\item Suppose that $\sur_{S_{\n}}(M_{\n},N_{\n})<\infty$ for some $\n\in Y$.  Then 
\[
\sur_S(M,N)\geqslant \min\{\sur_{S_{\m}}(M_{\m},N_{\m}):\m\in Y\}-\dim(Y).
\]
\end{enumerate}
\end{theorem}

Notice that Theorem~\ref{theorem:Bas} is concerned with the existence of a surjective $S$-linear map from $M$ to $S$.  If such a map exists, then it must split.  This observation motivates the following definition:

\begin{definition}\label{definition:spl}
Let $R$ be a commutative ring, $S$ an $R$-algebra, and $M$ and $N$ right $S$-modules.  We let $\spl_S(M,N)$ denote the supremum of the nonnegative integers $t$ such that there exists a split surjective $S$-linear map from $M$ to $N^{\oplus t}$, and we refer to $\spl_S(M,N)$ as the \textit{global splitting capacity of $M$ with respect to $N$ over $S$}.

Let $\p\in\Spec(R)$.  We refer to $\spl_{S_{\p}}(M_{\p},N_{\p})$ as the \textit{local splitting capacity of $M$ with respect to $N$ over $S$ at $\p$}.
\end{definition}

As with surjective capacities, we can always get an upper bound on a global splitting capacity in terms of local splitting capacities:
\[
\spl_S(M,N)\leqslant\inf\{\spl_{S_{\m}}(M_{\m},N_{\m}):\m\in\Max(R)\cap\Supp_R(N)\}.
\]

Next, we notice that, in Theorem~\ref{theorem:Bas}, we can replace every instance of the symbol $\sur$ with the symbol $\spl$ to get a lower bound on a global splitting capacity using local splitting capacities.  Our next theorem ensures that we can make the same modifications to Theorem~\ref{theorem:sur-Bas} if we assume that $N$ is finitely presented over $S$.  For emphasis, we mention that $N$ need not be flat over $S$ for the following theorem to hold.

\begin{theorem}\label{theorem:spl-Bas}
Let $R$ be a commutative ring, $S$ a module-finite $R$-algebra, $M$ a direct summand of a direct sum of finitely presented right $S$-modules, and $N$ a finitely presented right $S$-module.  Suppose that $Y:=\Max(R)\cap\Supp_R(N)$ is Noetherian with $\dim(Y)<\infty$.  Then the following statements hold:
\begin{enumerate}
\item Let $t$ be a positive integer, and suppose that $\spl_{S_{\m}}(M_{\m},N_{\m})\geqslant t+\dim(Y)$ for every $\m\in Y$.  Then $\spl_S(M,N)\geqslant t$.
\item $\spl_S(M,N)=\infty$ if and only if $\spl_{S_{\m}}(M_{\m},N_{\m})=\infty$ for every $\m\in Y$.
\item Suppose that $\spl_{S_{\n}}(M_{\n},N_{\n})<\infty$ for some $\n\in Y$.  Then 
\[
\spl_S(M,N)\geqslant \min\{\spl_{S_{\m}}(M_{\m},N_{\m}):\m\in Y\}-\dim(Y).
\]
\end{enumerate}
\end{theorem}

Now suppose that $R$ is a commutative Noetherian ring and that $M$ and $N$ are finitely generated $R$-modules.  In this more restricted setting, we can improve our lower bounds on global surjective and splitting capacities.  Our starting point for this work is Serre's Splitting Theorem from algebraic $K$-theory~\cite[\textit{Th\'{e}or\`{e}me~1}]{Ser}.  Below, $\dim_X(\p)$ refers to the Krull dimension of $\Var(\p)\cap X$, where $\p\in X\subseteq\Spec(R)$ and $\Var(\p):=\{\q\in\Spec(R):\p\subseteq\q\}$.  The symbol $j$-$\Spec(R)$ refers to the set of all $\p\in\Spec(R)$ such that $\p$ can be expressed as an intersection of maximal ideals of $R$.  The set $j$-$\Spec(R)$ was introduced by Swan in~\cite{Swa} and was used by Eisenbud and Evans in~\cite{EE} and by De~Stefani, Polstra, and Yao in~\cite{DSPY}.  In the corollary following Proposition~1 in~\cite{Swa}, Swan notes that $j$-$\Spec(R)$ is Noetherian if and only if $\Max(R)$ is Noetherian.  In the same corollary, Swan indicates that $\dim(j$-$\Spec(R))=\dim(\Max(R))$.

\begin{theorem}[Serre's Splitting Theorem~{\cite[\textit{Th\'{e}or\`{e}me~1}]{Ser}}]\label{theorem:Ser}
Let $R$ be a commutative Noetherian ring, and let $P$ be a finitely generated projective $R$-module.  Let $X:=j\textnormal{-}\Spec(R)$, and suppose that $\dim(X)<\infty$.  Then
\[
\sur_R(P,R)\geqslant\inf\{\sur_{R_{\p}}(P_{\p},R_{\p})-\dim_X(\p):\p\in X\}.
\]
\end{theorem}

The following result by De~Stefani, Polstra, and Yao generalizes Serre's Splitting Theorem by replacing $P$ with an arbitrary finitely generated $R$-module $M$:

\begin{theorem}[De~Stefani, Polstra, and Yao~{\cite[Theorem~3.12]{DSPY}}]\label{theorem:DSPY}
Let $R$ be a commutative Noetherian ring, and let $M$ be a finitely generated $R$-module.  Let $X:=j\textnormal{-}\Spec(R)$, and suppose that $\dim(X)<\infty$.  Then
\[
\sur_R(M,R)\geqslant\inf\{\sur_{R_{\p}}(M_{\p},R_{\p})-\dim_X(\p):\p\in X\}.
\]
\end{theorem}

Theorem~\ref{theorem:DSPY} can be compared with another extension of Serre's Splitting Theorem due to Stafford~\cite[Theorem~5.7]{Sta}.

We generalize Theorem~\ref{theorem:DSPY} in the following manner:  We let $S$ be a module-finite $R$-algebra; we allow $M$ to be a right $S$-module; and we replace the module $R$ with an arbitrary finitely generated $S$-module $N$:

\begin{theorem}\label{theorem:sur-Ser}
Let $R$ be a commutative Noetherian ring, $S$ a module-finite $R$-algebra, and $M$ and $N$ finitely generated right $S$-modules.  Let $X:=j\textnormal{-}\Spec(R)\cap\Supp_R(N)$, and suppose that $\dim(X)<\infty$.  Then
\[
\sur_S(M,N)\geqslant\inf\{\sur_{S_{\p}}(M_{\p},N_{\p})-\dim_X(\p):\p\in X\}.
\]
\end{theorem}

As with Theorem~\ref{theorem:Bas}, we can replace every instance of the symbol $\sur$ with the symbol $\spl$ in Theorems~\ref{theorem:Ser} and~\ref{theorem:DSPY}.  Our next theorem indicates that we can revise Theorem~\ref{theorem:sur-Ser} in the same way:

\begin{theorem}\label{theorem:spl-Ser}
Let $R$ be a commutative Noetherian ring, $S$ a module-finite $R$-algebra, and $M$ and $N$ finitely generated right $S$-modules.  Let $X:=j\textnormal{-}\Spec(R)\cap\Supp_R(N)$, and suppose that $\dim(X)<\infty$.  Then
\[
\spl_S(M,N)\geqslant\inf\{\spl_{S_{\p}}(M_{\p},N_{\p})-\dim_X(\p):\p\in X\}.
\]
\end{theorem}

We now describe the structure of our paper.  In Section~\ref{sec:sur-Ser}, we prove Theorem~\ref{theorem:sur-Ser}, modulo a lemma (Lemma~\ref{lemma:sur}).  We call this lemma the \textit{Surjective Lemma}, and we prove it in several stages:  Over the course of Sections~\ref{sec:sur-Bas} and~\ref{sec:Lambda}, we first reduce the proof of the Surjective Lemma to a verification that a certain condition holds on a finite set $\mathit{\Lambda}$ of prime ideals of a given commutative ring $R$.  Along the way, we prove Theorem~\ref{theorem:sur-Bas}, modulo the Surjective Lemma.  In Section~\ref{sec:Max}, we study the maximal ideals of $R$ in $\mathit{\Lambda}$ and continue working toward a proof of the Surjective Lemma.  In Section~\ref{sec:sur-lemma}, we address the remaining members of $\mathit{\Lambda}$ and complete a proof of the Surjective Lemma.  In Section~\ref{sec:spl}, we show that Theorems~\ref{theorem:spl-Bas} and~\ref{theorem:spl-Ser} can be proved by making a few key modifications to the proofs of Theorems~\ref{theorem:sur-Bas} and~\ref{theorem:sur-Ser}.  Finally, in Section~\ref{sec:char}, we consider the case of finitely generated modules over a Dedekind domain; in this case, we show that we can give conditions equivalent to having a given global surjective or splitting capacity.

In Sections~\ref{sec:sur-Ser}--\ref{sec:sur-lemma}, we abide by the following conventions:  We let $R$ be a commutative ring; we let $S$ be a module-finite $R$-algebra; we let $M$ denote a right $S$-module; and we let $N$ denote a finitely generated right $S$-module.  We view every left and right $S$-module as a standard $R$-module in the natural way.

\section{A Proof of Theorem~\texorpdfstring{\ref{theorem:sur-Ser}}{0.8}, modulo the Surjective Lemma}\label{sec:sur-Ser}

In this section, we prove Theorem~\ref{theorem:sur-Ser}, modulo the Surjective Lemma (Lemma~\ref{lemma:sur}).  The statement of the Surjective Lemma requires the following definition, which establishes analogues of local and global surjective capacities for arbitrary $R$-submodules of $\Hom_S(M,N)$:

\begin{definition}\label{definition:del}
Let $F$ be an $R$-submodule of $\Hom_S(M,N)$, and let $\p\in\Spec(R)$.  We let $\partial(F)$ denote the supremum of the nonnegative integers $t$ such that there exists a surjective $f\in F^{\oplus t}\subseteq \Hom_S(M,N^{\oplus t})$.  We let $\partial_{\p}(F)$ denote the supremum of the nonnegative integers $t$ such that there exists $f\in F^{\oplus t}$ with the property that $f_{\p}$ is surjective.

Let $n$ be a positive integer.  We will think of a member of $\Hom_S(M,N^{\oplus n})$ as a column
\[
\begin{pmatrix}
f_1 \\
\vdots \\
f_n \\
\end{pmatrix},
\]
where $f_1,\ldots,f_n\in \Hom_S(M,N)$.  When we refer to such an $n$-tuple without using a display, we write $(f_1,\ldots,f_n)^{\top}$ to denote the transpose of a row of functions.
\end{definition}

\begin{remark}\label{remark:infty}
Let $F$ be a finitely generated $R$-submodule of $\Hom_S(M,N)$.  We observe that $\partial(F)=\infty$ if and only if $N=0$:  Certainly, if $N=0$, then $\partial(F)=\infty$.  Suppose that $\partial(F)=\infty$, and let $n:=1+\mu_R(F)$.  Then there exist $f_1,\ldots,f_n\in F$ such that $F=Rf_1+\cdots+Rf_n$.  Let $f:=(f_1,\ldots,f_n)^{\top}$.  Since $\partial(F)=\infty$, there exists an $(n+1)\times n$ matrix $B$ with entries in $R$ such that $Bf$ is surjective.  Hence $B$ represents a surjective $R$-linear map from $N^{\oplus n}$ to $N^{\oplus (n+1)}$.  Now, for every $\m\in\Max(R)$, we see that $B\otimes 1_{R/\m}$ is surjective, which implies that $N_{\m}/\m_{\m} N_{\m} \cong N/\m N=0$ since $N$ is finitely generated over $R$.  Hence, for every $\m\in\Max(R)$, Nakayama's Lemma tells us that $N_{\m}=0$ since $N_{\m}$ is finitely generated over $R_{\m}$.  Thus $N=0$.

Let $\p\in\Spec(R)$.  Then, as a result of the discussion above, $\partial_{\p}(F)=\infty$ if and only if $\p\not\in\Supp_R(N)$.
\end{remark}

Using the symbol $\partial$, we now describe a condition weaker than surjectivity that we can impose on a member of $\Hom_S(M,N^{\oplus n})$.

\begin{definition}\label{definition:tXp}
Let $n,t$ be positive integers with $n\geqslant t$; let $\p\in X\subseteq\Spec(R)$; and let $f:=(f_1,\ldots,f_n)^{\top}\in\Hom_S(M,N^{\oplus n})$.  We say that $f$ is \textit{$(t,X,\p)$-surjective} if $\partial_{\p}(Rf_1+\cdots + Rf_n)\geqslant \min\{n,t+\dim_X(\p)\}$.

Let $Y\subseteq X$.  We say that $f$ is \textit{$(t,X,Y)$-surjective} if $f$ is $(t,X,\q)$-surjective for every $\q\in Y$.

When $t$ and $X$ are understood, we use the less cumbersome terms \textit{$\p$-surjective} and \textit{$Y$-surjective} in place of \textit{$(t,X,\p)$-surjective} and \textit{$(t,X,Y)$-surjective}, respectively. 
\end{definition}

\begin{remark}\label{remark:nXp}
Maintaining the hypotheses in the previous definition, we see that $f$ is $(n,X,\p)$-surjective if and only if $f_{\p}$ is surjective:  Suppose that $f$ is $(n,X,\p)$-surjective.  Then there exists an $n\times n$ matrix $B$ with entries in $R$ such that $B_{\p}f_{\p}=(Bf)_{\p}$ is surjective.  Now $B_{\p}$ represents a surjective $S_{\p}$-linear map from $N_{\p}^{\oplus n}$ to itself.  Since $N$ is finitely generated over $R$, we see that $N_{\p}$ is finitely generated over $R_{\p}$.  Hence $B_{\p}$ is bijective, and so $f_{\p}$ is surjective.  Conversely, if $f_{\p}$ is surjective, then $\partial_{\p}(Rf_1+\cdots + Rf_n)=n$, and so $f$ is $(n,X,\p)$-surjective.
\end{remark}

We need one more definition to state the Surjective Lemma, and this concerns the notion of a \textit{basic set for $R$}.  De~Stefani, Polstra, and Yao use basic sets to prove~\cite[Theorems~3.9 and~4.5]{DSPY}.  These theorems have conclusions that are weaker than those of~\cite[Theorems~3.12 and~4.8]{DSPY}, respectively, but the hypotheses of~\cite[Theorems~3.9 and~4.5]{DSPY} are more general.  Later in this section, we state an analogue of~\cite[Theorems~3.9 and~4.5]{DSPY} and a generalization of~\cite[Theorems~3.12 and~4.8]{DSPY} in Theorem~\ref{theorem:sur-prelim}.  We generalize~\cite[Theorems~3.9 and~4.5]{DSPY} in a separate paper~\cite{Bai2} where our goal is to extend Bass's Cancellation Theorem.

Here is the definition of a basic set for $R$, along with a few examples:

\begin{definition}\label{definition:basic}  
Let $X\subseteq\Spec(R)$.  We say that $X$ is a \textit{basic set for $R$} if $X$ is Noetherian and if, for every $\p\in\Spec(R)$ that can be written as an intersection of members of $X$, it is the case that $\p\in X$.
\end{definition}

\begin{example}
Every finite subset of $\Spec(R)$ is a basic set for $R$.
\end{example}

\begin{example}
If $R$ is Noetherian, then $\Spec(R)$ is a basic set for $R$.  If $R$ is a \textit{Jacobson ring}, then $j$-$\Spec(R)=\Spec(R)$.  Hence, if $R$ is a Noetherian Jacobson ring, then $j$-$\Spec(R)$ is a basic set for $R$.  In fact, a more general statement holds, as the next example shows.
\end{example}

\begin{example}\label{example:sur-X}
Let $X:=j$-$\Spec(R)\cap \Supp_R(N)$, and suppose that $X$ is Noetherian.  Then $X$ is a basic set for $R$:  Let $\p\in\Spec(R)$ such that $\p$ is an intersection of members of $X$.  Since every member of $j$-$\Spec(R)$ is an intersection of maximal ideals of $R$, so is $\p$.  Hence $\p\in j$-$\Spec(R)$.  Since $N$ is finitely generated over $R$, we see that $\Supp_R(N)=\Var(\Ann_R(N))$.  Hence $\p$ is an intersection of prime ideals that contain $\Ann_R(N)$, and so $\p$ must itself contain $\Ann_R(N)$.  Thus $\p\in\Var(\Ann_R(N))=\Supp_R(N)$.  We have proved then that $\p\in X$.  Thus $X$ is a basic set for $R$.
\end{example}

\begin{example}
By Example~\ref{example:sur-X}, if $j$-$\Spec(R)$ is Noetherian, then $j$-$\Spec(R)$ is a basic set for~$R$.  If $R$ is Artinian, then $\Max(R)=j$-$\Spec(R)=\Spec(R)$.  If $R$ is \textit{semilocal} (Noetherian with only finitely many maximal ideals) but not Artinian, then $\Max(R)=j$-$\Spec(R)\subsetneq \Spec(R)$, and $\dim(j$-$\Spec(R))=0<\dim(\Spec(R))$.
\end{example}

\begin{example}
Suppose that $R$ is a one-dimensional Noetherian domain with infinitely many maximal ideals.  Then $\Max(R)$ is not a basic set for $R$:  If $\Max(R)$ is a basic set for $R$, then $\Jac(R)\neq 0$, and so $\Min(\Jac(R))=\Max(R)$ is a finite set, a contradiction.

Notice that, in this example, $\Max(R)\subsetneq j$-$\Spec(R)=\Spec(R)$.
\end{example}

\begin{example}
Let $d\in\ZZ$ with $d\geqslant 2$.  Let $T$ be the domain
\[
K[x_1,\ldots,x_d,y_1,y_2,y_3,\ldots],
\]
where $K$ is a field and $x_1,\ldots,x_d,y_1,y_2,y_3,\ldots$ are indeterminates.  Invert every element of $T$ outside of the set
\[
(Tx_1+\cdots+Tx_d)\cup Ty_1\cup Ty_2\cup Ty_3 \cup \cdots
\]
to form a new domain, and suppose that $R$ is this new domain.  Example~3.3 in~\cite{Olb} then indicates that $R$ is a $d$-dimensional Noetherian ring such that
\[
\Max(R)=\{Rx_1+\cdots+Rx_d, Ry_1, Ry_2, Ry_3, \ldots\}
\]
and such that every nonzero prime ideal of $R$ is contained a unique maximal ideal.  We will prove that $\Max(R)\subsetneq j$-$\Spec(R)\subsetneq\Spec(R)$.

To prove that the first inclusion is strict, we will show that $0\in j$-$\Spec(R)-\Max(R)$.  Of course, $0\not\in\Max(R)$.  Suppose that $0\not\in j$-$\Spec(R)$.  Then $\Jac(R)\neq 0$.  Since $R$ is Noetherian, $\Min(\Jac(R))$ is a finite set.  On the other hand, $\Min(\Jac(R))$ contains the infinite set $\{Ry_1, Ry_2, Ry_3, \ldots\}$, a contradiction.  Hence $0\in j$-$\Spec(R)$.

To prove that $j$-$\Spec(R)\neq\Spec(R)$, we will first show that $j$-$\Spec(R)=\{0\}\cup\Max(R)$.  We have already shown that $0\in j$-$\Spec(R)$, and $\Max(R)\subseteq j$-$\Spec(R)$ by definition, so it remains to show that every nonzero member $\p$ of $j$-$\Spec(R)$ is in $\Max(R)$.  As we mentioned above, $\p$ is contained in a unique maximal ideal of $R$.  Since $\p$ is an intersection of maximal ideals of $R$, it must then be the case that $\p\in\Max(R)$.  Hence $j$-$\Spec(R)=\{0\}\cup\Max(R)$, and so $Rx_1\in\Spec(R)-j$-$\Spec(R)$.  Thus $j$-$\Spec(R)\neq\Spec(R)$.

Notice that, in this example, $\dim(j$-$\Spec(R))=1<d=\dim(\Spec(R))$.  In particular, if $X:=\Spec(R)$ and $Y:=j$-$\Spec(R)$, then $\dim_Y(0)=1<d=\dim_X(0)$.
\end{example}

We are now prepared to state the Surjective Lemma.  This lemma is similar in spirit to~\cite[Section~8, Lemma~II]{Bas};~\cite[Theorem~2.4]{CLQ};~\cite[Lemmas~3.8 and~4.4]{DSPY};~\cite[Lemma~3, Lemma~5, and Theorem~B]{EE};~\cite[Theorem~2.1]{Hei};~\cite[\textit{Th\'{e}or\`{e}me~2}]{Ser}; and~\cite[Lemma~5.4]{Sta}, although several techniques that we use to prove the Surjective Lemma, as we will see, are decidedly distinct from previous work.

\begin{lemma}[Surjective Lemma]\label{lemma:sur}
Let $n,t$ be positive integers with $n\geqslant 1+t$, and let $X$ be a subset of $\Supp_R(N)$ that is a basic set for $R$.  Let $f:=(f_1,\ldots,f_n)^{\top}\in\Hom_S(M,N^{\oplus n})$, and suppose that $f$ is $(t,X,X)$-surjective.  Then there exist $f'_1,\ldots,f'_{n-1}\in Rf_1+\cdots + Rf_n$ such that $f':=(f'_1,\ldots, f'_{n-1})^{\top}$ is $(t,X,X)$-surjective.
\end{lemma}

We prove this lemma in Section~\ref{sec:sur-lemma}.  Given this lemma, we can prove Theorem~\ref{theorem:sur-prelim} below.  Part~(1) of Theorem~\ref{theorem:sur-prelim} provides an analogue of~\cite[Theorems~3.9 and~4.5]{DSPY}; and Part~(2) generalizes~\cite[Theorems~3.12 and~4.8]{DSPY} and~\cite[\textit{Th\'{e}or\`{e}me~1}]{Ser}.  Theorems that extend~\cite[\textit{Th\'{e}or\`{e}me~1}]{Ser} in other directions include~\cite[Corollary~3.2]{CLQ};~\cite[Theorem~A]{EE};~\cite[Theorem~2.5]{Hei}; and~\cite[Theorem~5.7]{Sta}.

\begin{theorem}\label{theorem:sur-prelim}
Let $L$ be an $S$-submodule of $M$; let $F$ be a finitely generated $R$-submodule of $\Hom_S(L,N)$; and let $G$ be an $R$-submodule of $\Hom_S(M,N)$.  Suppose that every member of $F$ can be extended to a member of $G$.  Let $X$ be a subset of $\Supp_R(N)$ that is a basic set for $R$, and suppose that $\dim(X)<\infty$.  Then the following statements hold:
\begin{enumerate}
\item Let $t$ be a positive integer, and suppose that $\partial_{\p}(F)\geqslant t+\dim_X(\p)$ for every $\p\in X$.  Then there exists $g\in G^{\oplus t}$ such that $g_{\p}$ is surjective for every $\p\in X$.
\item Suppose that $\Max(R)\cap\Supp_R(N)\subseteq X$.  Then
\[
\partial(G)\geqslant\inf\{\partial_{\p}(F)-\dim_X(\p):\p\in X\}.
\]
\end{enumerate}
\end{theorem}

\begin{proof}[Proof of Theorem~\ref{theorem:sur-prelim}, modulo the Surjective Lemma]
(1)  We may assume that $X\neq\varnothing$.  Let $n:=\mu_R(F)$.  Since $t$ is a positive integer, $n$ is a positive integer.  Let $f_1,\ldots,f_n\in F$ such that $F=Rf_1+\cdots+Rf_n$; let $f:=(f_1,\ldots,f_n)^{\top}$; and let $\q\in X$.  Then there exists a $t\times n$ matrix $B$ with entries in $R$ such that $(Bf)_{\q}=B_{\q}f_{\q}$ is surjective.  Hence $B\otimes 1_{\kappa(\q)}$ represents a surjective $\kappa(\q)$-linear map from $N^{\oplus n}\otimes \kappa(\q)$ to $N^{\oplus t}\otimes \kappa(\q)$.  Since $N\otimes \kappa(\q)$ is finitely generated over $\kappa(\q)$, we see that $n\geqslant t$.  Thus $f$ is $(t,X,X)$-surjective.  Now, after $n-t$ applications of the Surjective Lemma, we obtain $f'_1,\ldots,f'_t\in F$ such that $f':=(f'_1,\ldots, f'_t)^{\top}$ is $(t,X,X)$-surjective.  By Remark~\ref{remark:nXp}, we see that $f'_{\p}$ is surjective for every $\p\in X$.  By assumption, the members $f'_1,\ldots, f'_t$ of $F$ can be extended to members $g_1,\ldots,g_t$ of $G$, respectively.  Let $g:=(g_1,\ldots,g_t)^{\top}\in G^{\oplus t}$.  Since $f'_{\p}$ is surjective for every $\p\in X$, we see that $g_{\p}$ is surjective for every $\p\in X$.

(2)  Let 
\[
t:=\inf\{\partial_{\p}(F)-\dim_X(\p):\p\in X\}.
\]
If $t\leqslant 0$, then there is nothing to prove.  If $t=\infty$, then $N=0$ by Remark~\ref{remark:infty}, and so $\partial(G)=t$.  Suppose then that $t$ is a positive integer.  By Part~(1) of this theorem, there exists $g\in G^{\oplus t}$ such that $g_{\p}$ is surjective for every $\p\in X$.  Since $\Max(R)\cap\Supp_R(N)\subseteq X$, we see that $g$ is surjective.  Hence $\partial(G)\geqslant t$.
\end{proof}

We now show that Theorem~\ref{theorem:sur-Ser} is an easy corollary of Theorem~\ref{theorem:sur-prelim}.

\begin{proof}[Proof of Theorem~\ref{theorem:sur-Ser}, modulo the Surjective Lemma]
Let $L=M$, and let $F=G$ denote $\Hom_S(M,N)$.  By our hypotheses, $F$ is finitely generated over $R$.  By Example~\ref{example:sur-X}, we see that $X$ is a subset of $\Supp_R(N)$ that is a basic set for $R$.  Also, $\Max(R)\cap\Supp_R(N)\subseteq X$.  Hence, by Part~(2) of Theorem~\ref{theorem:sur-prelim}, we get
\[
\begin{array}{rcl}
\sur_S(M,N) & = & \partial(G) \\
& \geqslant & \inf\{\partial_{\p}(F)-\dim_X(\p):\p\in X\} \\
& = & \inf\{\sur_{S_{\p}}(M_{\p},N_{\p})-\dim_X(\p):\p\in X\}. \qedhere \\
\end{array}
\]
\end{proof}

With a little more work, we can also prove Theorem~\ref{theorem:sur-Bas}, modulo the Surjective Lemma.  We accomplish this goal in the next section.

\section{A Proof of Theorem~\texorpdfstring{\ref{theorem:sur-Bas}}{0.3}, modulo the Surjective Lemma}\label{sec:sur-Bas}

Under the hypotheses of Theorem~\ref{theorem:sur-Ser}, we see that $\Hom_S(M,N)$ is finitely generated over $R$, and we see that $[\Hom_S(M,N)]_{\p}$ and $\Hom_{S_{\p}}(M_{\p},N_{\p})$ are isomorphic as $R_{\p}$-modules for every $\p\in\Spec(R)$.  Hence, a finite number of global $S$-linear maps govern all local surjective capacities.

This property does not follow from the more general hypotheses of Theorem~\ref{theorem:sur-Bas}.  As a result, we are not immediately in a position to apply the Surjective Lemma (Lemma~\ref{lemma:sur}) when attempting to prove Theorem~\ref{theorem:sur-Bas}.

In this section, we show that we can circumvent the issue by obtaining a finitely generated $R$-submodule $F$ of $\Hom_S(M,N)$ such that $\partial_{\m}(F)$ is sufficiently large for every $\m\in\Max(R)\cap\Supp_R(N)$.  We can then prove Theorem~\ref{theorem:sur-Bas}, modulo the Surjective Lemma.

With this goal in mind, we first show that, for every $R$-submodule $F$ of $\Hom_S(M,N)$, the function on $\Spec(R)$ taking $\p$ to $\partial_{\p}(F)$ is lower semicontinuous.  This lemma can be compared to~\cite[Lemma~7.2]{Bas};~\cite[Lemmas~3.5 and~4.1]{DSPY};~\cite[Theorem~1.5]{Hei};~\cite[\textit{Lemme~3}]{Ser}; and~\cite[Proposition~3]{Swa}.

\begin{lemma}\label{lemma:sur-closed}
Let $F$ be an $R$-submodule of $\Hom_S(M,N)$, and let $t$ be a nonnegative integer.  Then the set $\{\p\in\Spec(R):\partial_{\p}(F)> t\}$ is open, and so the set $\{\p\in\Spec(R):\partial_{\p}(F)\leqslant t\}$ is closed.  Hence, for every subspace $X$ of $\Spec(R)$, the set $Y_t:=\{\p\in X:\partial_{\p}(F)\leqslant t\}$ is closed in $X$.
\end{lemma}

\begin{proof}
Let $\p\in\Spec(R)$ such that $\partial_{\p}(F)>t$; let $f\in F^{\oplus (t+1)}$ such that $f_{\p}$ is surjective; and let $C:=\coker f$.  Since $C_{\p}=0$ and since $C$ is finitely generated over $R$, there is an element $s\in R-\p$ such that $sC=0$.  Let $U:=\{\q\in\Spec(R): s\not\in\q\}$.  Then $C_{\q}=0$ for every $\q\in U$.  Hence $U$ is an open neighborhood of $\p$ such that $\partial_{\q}(F)>t$ for every $\q\in U$. Thus the set $\{\p\in\Spec(R): \partial_{\p}(F)>t\}$ is open.  This proves the first claim of the lemma.  The last two claims of the lemma follow from the first claim.
\end{proof}

Using the previous lemma, we can show that, under the hypotheses of Part~(1) of Theorem~\ref{theorem:sur-Bas}, there exists a finitely generated $R$-submodule $F$ of $\Hom_S(M,N)$ such that $\partial_{\m}(F)\geqslant t+\dim(Y)$ for every $\m\in Y$.  This is, in fact, a consequence of a more general statement:

\begin{lemma}\label{lemma:sur-fg}
Suppose that $M$ is a direct summand of a direct sum of finitely presented right $S$-modules.  Let $X$ be a Noetherian subspace of $\Supp_R(N)$, and suppose that $\dim(X)<\infty$.  Let $t$ be a positive integer, and suppose that $\sur_{S_{\p}}(M_{\p},N_{\p})\geqslant t+\dim(X)$ for every $\p\in X$.  Then there exists a finitely generated $R$-submodule $F$ of $\Hom_S(M,N)$ such that $\partial_{\p}(F)\geqslant t+\dim(X)$ for every $\p\in X$.
\end{lemma}

\begin{proof}
Let $u:=(t-1)+\dim(X)$, and let $\mathscr{G}$ denote the collection of all finitely generated $R$-submodules of $\Hom_S(M,N)$.  For every $G\in\mathscr{G}$, let $Y(G):=\{\p\in X:\partial_{\p}(G)\leqslant u\}$, and let $\mathscr{Y}:=\{Y(G):G\in\mathscr{G}\}$.  We aim to prove that $\varnothing\in\mathscr{Y}$.  Suppose not.  Then $X$ is nonempty.  By Lemma~\ref{lemma:sur-closed}, we see that $Y(G)$ is closed in $X$ for every $G\in\mathscr{G}$.  Since $X$ is Noetherian and nonempty, there is $G'\in\mathscr{G}$ such that $Y(G')$ is a minimal member of~$\mathscr{Y}$.  By assumption, $Y(G')\neq\varnothing$, so let $\q\in Y(G')$.  By our hypothesis on $M$, there exist $h_1,\ldots,h_{u+1}\in\Hom_S(M,N)$ such that $h:=(h_1,\ldots,h_{u+1})^{\top}$ has the property that $h_{\q}$ is surjective.  Let $H:=G'+Rh_1+\cdots+Rh_{u+1}$.  Then $H\in\mathscr{G}$, and $\partial_{\q}(H)\geqslant u+1=t+\dim(X)$.  Thus $\q\in Y(G')-Y(H)$, and so $Y(H)\subsetneq Y(G')$, contradicting the minimality of $Y(G')$ in~$\mathscr{Y}$.  Thus $\varnothing\in\mathscr{Y}$, and so there exists a finitely generated $R$-submodule $F$ of $\Hom_S(M,N)$ such that $Y(F)=\varnothing$.  Consequently, $\partial_{\p}(F)\geqslant u+1=t+\dim(X)$ for every $\p\in X$. 
\end{proof}

We now prove Theorem~\ref{theorem:sur-Bas}, modulo the Surjective Lemma.

\begin{proof}[Proof of Theorem~\ref{theorem:sur-Bas}, modulo the Surjective Lemma]
Let $X:=j$-$\Spec(R)\cap\Supp_R(N)$.  By the corollary following Theorem~1 in~\cite{Swa}, we see that $\dim(X)=\dim(Y)<\infty$ and that $X$ is Noetherian since $Y$ is Noetherian.  By Example~\ref{example:sur-X}, we see that $X$ is a basic set for $R$.

(1)  Since $\sur_{S_{\m}}(M_{\m},N_{\m})\geqslant t+\dim(Y)$ for every $\m\in Y$, we see that $\sur_{S_{\p}}(M_{\p},N_{\p})\geqslant t+\dim(X)$ for every $\p\in X$.  By Lemma~\ref{lemma:sur-fg}, there exists a finitely generated $R$-submodule $F$ of $\Hom_S(M,N)$ such that $\partial_{\p}(F)\geqslant t+\dim(X)\geqslant t+\dim_X(\p)$ for every $\p\in X$.  Now, let $L=M$ and $G=\Hom_S(M,N)$.  Then Part~(2) of Theorem~\ref{theorem:sur-prelim} tells us that $\sur_S(M,N)=\partial(G)\geqslant t$.

(2) Certainly, if $\sur_S(M,N)=\infty$, then $\sur_{S_{\m}}(M_{\m},N_{\m})=\infty$ for every $\m\in Y$.  Suppose then that $\sur_{S_{\m}}(M_{\m},N_{\m})=\infty$ for every $\m\in Y$.  Then, $\sur_{S_{\m}}(M_{\m},N_{\m})\geqslant t+\dim(Y)$ for every $\m\in Y$ and for every positive integer $t$.  Hence, by Part~(1) of this theorem, $\sur_S(M,N)\geqslant t$ for every positive integer $t$.  Thus $\sur_S(M,N)=\infty$.

(3)  Let
\[
t:=\min\{\sur_{S_{\m}}(M_{\m},N_{\m}):\m\in Y\}-\dim(Y).
\]
Then $\sur_{S_{\m}}(M_{\m},N_{\m})\geqslant t+\dim(Y)$ for every $\m\in Y$.  If $t\leqslant 0$, then certainly $\sur_S(M,N)\geqslant t$.  Suppose then that $t$ is a positive integer.  Then, by Part~(1) of this theorem, $\sur_S(M,N)\geqslant~t$.
\end{proof}

We can also prove the following variations of Lemma~\ref{lemma:sur-fg} and Theorems~\ref{theorem:sur-prelim} and~\ref{theorem:sur-Bas}.  These variations are noteworthy in the sense that they do not require $M$ to be a direct summand of a direct sum of finitely presented right $S$-modules.  We omit the proofs since they are similar to the proofs of Lemma~\ref{lemma:sur-fg} and Theorems~\ref{theorem:sur-prelim} and~\ref{theorem:sur-Bas}.

\begin{lemma}\label{lemma:sur-fg-var}
Let $F$ be an $R$-submodule of $\Hom_S(M,N)$.  Let $X$ be a Noetherian subspace of $\Supp_R(N)$, and suppose that $\dim(X)<\infty$.  Let $t$ be a positive integer, and suppose that $\partial_{\p}(F)\geqslant t+\dim(X)$ for every $\p\in X$.  Then there exists a finitely generated $R$-submodule $F'$ of $F$ such that $\partial_{\p}(F')\geqslant t+\dim(X)$ for every $\p\in X$. 
\end{lemma}

\begin{theorem}\label{theorem:sur-Bas-var}
Let $L$ be an $S$-submodule of $M$; let $F$ be an $R$-submodule of $\Hom_S(L,N)$; and let $G$ be an $R$-submodule of $\Hom_S(M,N)$.  Suppose that every member of $F$ can be extended to a member of $G$.  Then the following statements hold:
\begin{enumerate}
\item Let $X$ be a subset of $\Supp_R(N)$ that is a basic set for $R$ with $\dim(X)<\infty$.  Let $t$ be a positive integer, and suppose that $\partial_{\p}(F)\geqslant t+\dim(X)$ for every $\p\in X$.  Then there exists $g\in G^{\oplus t}$ such that $g_{\p}$ is surjective for every $\p\in X$.
\item Suppose that $Y:=\Max(R)\cap\Supp_R(N)$ is Noetherian with $\dim(Y)<\infty$.  Then the following statements hold:
\begin{enumerate}
\item Let $t$ be a positive integer, and suppose that $\partial_{\m}(F)\geqslant t+\dim(Y)$ for every $\m\in Y$.  Then $\partial(G)\geqslant t$.
\item If $\partial_{\m}(F)=\infty$ for every $\m\in Y$, then $\partial(G)=\infty$.  Hence $\partial(F)=\infty$ if and only if $\partial_{\p}(F)=\infty$ for every $\m\in Y$.
\item Suppose that $\partial_{\n}(F)<\infty$ for some $\n\in Y$.  Then
\[
\partial(G)\geqslant\min\{\partial_{\m}(F):\m\in Y\}-\dim(Y).
\]
\end{enumerate}
\end{enumerate}
\end{theorem}

In the next section, we continue working toward a proof of the Surjective Lemma.

\section{The Set \texorpdfstring{$\mathit{\Lambda}$}{Lambda}}\label{sec:Lambda}

Assuming the hypotheses of the Surjective Lemma (Lemma~\ref{lemma:sur}), we show in this section that there is a finite subset $\mathit{\Lambda}$ of $X$ with a crucial property:  If there exists an invertible $n\times n$ matrix $A$ with entries in $R$ such that the first $n-1$ components of $Af:=(f'_1,\ldots,f'_n)^{\top}$ form a map $f':=(f'_1,\ldots,f'_{n-1})^{\top}$ that is $(t,X,\mathit{\Lambda})$-surjective, then $f'$ is $(t,X,X)$-surjective.  In other words, we show that proving the Surjective Lemma, which may involve infinitely many prime ideals of $R$, reduces to proving a statement about $\mathit{\Lambda}$, a finite set of prime ideals.

We begin by covering two properties of basic sets for $R$ with Proposition~\ref{proposition:basic}.  Property~(1), which holds not only for every basic set but also for every Noetherian topological space, is proved, for example, in~\cite[Proposition~1.5]{Har}.  Property~(2) is proved for the special case of $X:=j$-$\Spec(R)$ in~\cite[Proposition~2]{Swa}, although the more general result here can be proved by similar means.

\begin{proposition}\label{proposition:basic}
Let $X$ be a basic set for $R$.  Then $X$ has the following properties:
\begin{enumerate}
\item Every closed subset of $X$ is a union of finitely many closed irreducible subsets of $X$.
\item Every closed irreducible subset of $X$ has a unique generic point.
\end{enumerate}
\end{proposition}

In the next lemma, we define the set $\mathit{\Lambda}$ mentioned earlier in this section.  We omit the proof of this lemma since it is similar to the proofs of~\cite[Lemmas~3.6 and~4.2]{DSPY}.

\begin{lemma}\label{lemma:sur-Lambda}
Let $F$ be a finitely generated $R$-submodule of $\Hom_S(M,N)$, and let $X$ be a subset of $\Supp_R(N)$ that is a basic set for $R$.  Then there exists a finite subset $\mathit{\Lambda}$ of $X$ such that, for every $\p\in X-\mathit{\Lambda}$, there exists $\q\in\mathit{\Lambda}$ with the properties that $\q\subsetneq \p$ and $\partial_{\q}(F)=\partial_{\p}(F)$.
\end{lemma}

Assume the hypotheses of Part~(2) of Theorem~\ref{theorem:sur-prelim}, and define $\mathit{\Lambda}$ as in Lemma~\ref{lemma:sur-Lambda} with respect to $F$ and $X$.  In Part~(2) of Theorem~\ref{theorem:sur-prelim}, we give a lower bound on $\partial(G)$ whose expression involves all of the members of $X$, where $X$ could be an infinite set.  In Part~(1) of the following corollary, we express this same lower bound using only the members of the finite set $\mathit{\Lambda}$.  As a result, we can improve our expression of the lower bound on $\sur_S(M,N)$ in Theorem~\ref{theorem:sur-Ser} as well, and this is the content of Part~(2) of the following corollary.  Of course, since Theorems~\ref{theorem:sur-prelim} and~\ref{theorem:sur-Ser} rely on the Surjective Lemma, we still need to assume the truth of the Surjective Lemma for the following corollary.

\begin{corollary}\label{corollary:sur-prelim}
We make the following improvements to Theorems~\ref{theorem:sur-prelim} and~\ref{theorem:sur-Ser}:
\begin{enumerate}
\item Assume the hypotheses of Part~(2) of Theorem~\ref{theorem:sur-prelim}, and let $\mathit{\Lambda}$ be defined as in Lemma~\ref{lemma:sur-Lambda} with respect to $F$ and $X$.  Then
\[
\partial(G)\geqslant\inf\{\partial_{\p}(F)-\dim_X(\p):\p\in\mathit{\Lambda}\}.
\]
\item Assume the hypotheses of Theorem~\ref{theorem:sur-Ser}, and let $\mathit{\Lambda}$ be defined as in Lemma~\ref{lemma:sur-Lambda} with respect to $F:=\Hom_S(M,N)$ and $X$.  Then
\[
\sur_S(M,N)\geqslant\inf\{\sur_{S_{\p}}(M_{\p},N_{\p})-\dim_X(\p):\p\in\mathit{\Lambda}\}.
\]
\end{enumerate}
\end{corollary}

\begin{proof}[Proof of Corollary~\ref{corollary:sur-prelim}, modulo the Surjective Lemma]
(1)  Let
\[
t:=\inf\{\partial_{\p}(F)-\dim_X(\p):\p\in X\},
\]
and let
\[
u:=\inf\{\partial_{\p}(F)-\dim_X(\p):\p\in\mathit{\Lambda}\}.
\]
We will show that $t=u$.  Certainly $t\leqslant u$, so it remains to show that $t\geqslant u$.  If $t=\infty$, then certainly $t\geqslant u$.  Suppose then that $t$ is an integer so that $X\neq\varnothing$.  Let $\p_0\in X$ such that $t=\partial_{\p_0}(F)-\dim_X(\p_0)$.  Suppose, by way of contradiction, that $\p_0\not\in\mathit{\Lambda}$.  By Lemma~\ref{lemma:sur-Lambda}, there exists $\q_0\in\mathit{\Lambda}$ such that $\q_0\subsetneq\p_0$ and $\partial_{\q_0}(F)=\partial_{\p_0}(F)$.  Hence 
\[
t=\partial_{\p_0}(F)-\dim_X(\p_0)>\partial_{\q_0}(F)-\dim_X(\q_0)\geqslant t,
\]
a contradiction.  Thus $\p_0\in\mathit{\Lambda}$, and so $t\geqslant u$.  Thus $t=u$.  Now, by Part~(2) of Theorem~\ref{theorem:sur-prelim}, we see that $\partial(G)\geqslant t=u$.

(2)  Let $L=M$, and let $G=F$.  Since $F$ is finitely generated over $R$, Part~(1) of this corollary tells us that
\[
\begin{array}{rcl}
\sur_S(M,N) & = & \partial(G) \\
& \geqslant & \inf\{\partial_{\p}(F)-\dim_X(\p):\p\in\mathit{\Lambda}\} \\
& = & \inf\{\sur_{S_{\p}}(M_{\p},N_{\p})-\dim_X(\p):\p\in\mathit{\Lambda}\}. \qedhere \\
\end{array}
\]
\end{proof}

We now return to the task of reducing the proof of the Surjective Lemma to the study of a finite subset $\mathit{\Lambda}$ of $X$.  The following definition will be useful in subsequent work:

\begin{definition}
Let $n$ be a positive integer.  The symbol $\GL(n,R)$ refers to the group of all invertible $n\times n$ matrices with entries in $R$.  This group is called the \textit{general linear group of degree $n$ over $R$}.
\end{definition}

\begin{remark}
With respect to the previous definition, an $n\times n$ matrix with entries in $R$ is in $\GL(n,R)$ if and only if its determinant is a unit of $R$.
\end{remark}

In the following definition, we establish a large amount of notation that we will use in the remainder of this section and in the next three sections:

\begin{definition}\label{definition:sur-F}
Let $n\in\ZZ$ with $n\geqslant 2$; let $f:=(f_1,\ldots,f_n)^{\top}\in \Hom_S(M,N^{\oplus n})$; and let $\p\in\Supp_R(N)$.  Fix the following notation relative to $n$, $f$, and $\p$:  Let $F:=Rf_1+\cdots+Rf_n$; let $^-$ denote the functor $-\otimes_R \kappa(\p)$; and, for every matrix $\mathit{\Xi}:=(\xi_{i,j})$ with entries in $R$ or $R_{\p}$, let $\overline{\mathit{\Xi}}:=\left(\overline{\xi_{i,j}}\right)$.  If a matrix $A\in \GL(n,R)$ is given, then let $f^*:=Af:=(f'_1,\ldots,f'_n)^{\top}$; let $f':=(f'_1,\ldots,f'_{n-1})^{\top}$; and let $F':=Rf'_1+\cdots +Rf'_{n-1}$. 
\end{definition}

Upon proving the following lemma, we will be prepared to achieve the goal of this section.

\begin{lemma}\label{lemma:sur-del-one}
Let $n\in\ZZ$ with $n\geqslant 2$; let $f:=(f_1,\ldots,f_n)^{\top}$ $\in \Hom_S(M,N^{\oplus n})$; let $A\in \textnormal{\textbf{GL}}(n,R)$; and let $\p\in\Supp_R(N)$.  Then, relative to Definition~\ref{definition:sur-F}, we have $\partial_{\p}(F')\geqslant \partial_{\p}(F)-1$.
\end{lemma}

\begin{proof}
If $\partial_{\p}(F)\leqslant 1$, then $\partial_{\p}(F')\geqslant 0\geqslant\partial_{\p}(F)-1$, and so we are done.  Suppose then, for the rest of the proof, that $\partial_{\p}(F)\geqslant 2$.  

Let $d:=\partial_{\p}(F)$, and let $B$ be a $d\times n$ matrix with entries in $R$ such that $(Bf)_{\p}$ is surjective.  Let $C\in \GL(d,R_{\p})$ such that $CBA^{-1}$ can be represented as a matrix $(b_{i,j})$ with entries in $R$, where $b_{1,n}, \ldots, b_{d-1,n}\in\p$.  Hence
\[
\overline{CBA^{-1}}
=\begin{pmatrix}
\overline{b_{1,1}} & \cdots & \overline{b_{1,n-1}} & \overline{0} \\
\vdots & \ddots & \vdots & \vdots \\
\overline{b_{d-1,1}} & \cdots & \overline{b_{d-1,n-1}} & \overline{0} \\
\overline{b_{d,1}} & \cdots & \overline{b_{d,n-1}} & \overline{b_{d,n}} \\
\end{pmatrix}.
\]
Let
\[
B'
:=\begin{pmatrix}
b_{1,1} & \cdots & b_{1,n-1} \\
\vdots & \ddots & \vdots \\
b_{d-1,1} & \cdots & b_{d-1,n-1} \\
\end{pmatrix}.
\]
Then $B'f'\in(F')^{\oplus (d-1)}$, and $\overline{B'f'}$ is surjective.  Nakayama's Lemma then tells us that $(B'f')_{\p}$ is surjective.  Hence $\partial_{\p}(F')\geqslant d-1= \partial_{\p}(F)-1$.
\end{proof}

\begin{lemma}\label{lemma:sur-reduction}
Assume the hypotheses of the Surjective Lemma, and define $\mathit{\Lambda}$ as in Lemma~\ref{lemma:sur-Lambda} with respect to $F:=Rf_1+\cdots+Rf_n$ and $X$.  Let $A\in \textnormal{\textbf{GL}}(n,R)$, and suppose that, with respect to Definition~\ref{definition:sur-F}, we have that $f'$ is $(t,X,\mathit{\Lambda})$-surjective.  Then $f'$ is $(t,X,X)$-surjective.
\end{lemma}

\begin{proof}
Since $t$ and $X$ are understood, we can use the terms \textit{$\p$-surjective} and \textit{$Y$-surjective} for any $\p\in X$ and for any $Y\subseteq X$ without the risk of confusion.  Hence, we aim to prove that $f'$ is $X$-surjective.

Let $\p\in X-\mathit{\Lambda}$.  By Lemma~\ref{lemma:sur-del-one}, we have $\partial_{\p}(F')\geqslant \partial_{\p}(F)-1$.  By Lemma~\ref{lemma:sur-Lambda}, there exists $\q\in \mathit{\Lambda}$ such that $\q\subsetneq \p$ and $\partial_{\q}(F)= \partial_{\p}(F)$.  Since $f$ is $\q$-surjective by assumption, we now have
\[
\begin{array}{rcl}
\partial_{\p}(F') & \geqslant & \partial_{\p}(F)-1 \\
& = & \partial_{\q}(F)-1 \\
& \geqslant & \min\{n,t+\dim_X(\q)\}-1 \\
& \geqslant & \min\{n-1,t+\dim_X(\p)\}, \\
\end{array}
\]
and so $f'$ is $\p$-surjective.  Thus $f'$ is $(X-\mathit{\Lambda})$-surjective.  Since $f'$ is $\mathit{\Lambda}$-surjective by assumption, we conclude that $f'$ is $X$-surjective.
\end{proof}

Assume the hypotheses of the Surjective Lemma.  By Lemma~\ref{lemma:sur-reduction}, proving the Surjective Lemma now reduces to finding a matrix $V\in \GL(n,R)$ such that the first $n-1$ components of $Vf:=(g_1,\ldots,g_n)^{\top}$ form a map $(g_1,\ldots,g_{n-1})^{\top}$ that is $(t,X,\mathit{\Lambda})$-surjective, where $\mathit{\Lambda}$ is defined as in Lemma~\ref{lemma:sur-Lambda} with respect to $F:=Rf_1+\cdots+Rf_n$ and $X$.  We compute such a matrix $V$ over the course of the next two sections, completing a proof of the Surjective Lemma in Section~\ref{sec:sur-lemma}.

\section{The Maximal Ideals of \texorpdfstring{$R$}{R} in \texorpdfstring{$\mathit{\Lambda}$}{Lambda}}\label{sec:Max}

Assume the hypotheses of the Surjective Lemma, and define $\mathit{\Lambda}$ as in Lemma~\ref{lemma:sur-Lambda} with respect to $F:=Rf_1+\cdots+Rf_n$ and $X$.  In this section, we find a matrix $V\in\GL(n,R)$ such that the first $n-1$ components of $Vf:=(g_1,\ldots,g_n)^{\top}$ form a map $(g_1,\ldots,g_{n-1})^{\top}$ that is $(t,X,\m)$-surjective for every $\m\in\mathit{\Lambda}\cap\Max(R)$.  We accomplish this goal over the course of three lemmas.  The following definition will be useful in the work to come:

\begin{definition}
Let $n\in\ZZ$ with $n\geqslant 2$.  For every $i\in\{1,\ldots,n-1\}$, let $P_i$ be the $n\times n$ permutation matrix obtained by switching the $i$th row and the $n$th row of the $n\times n$ identity matrix, and let $P_n$ denote the $n\times n$ identity matrix itself.
\end{definition}

We begin with the following lemma:

\begin{lemma}\label{lemma:sur-L}
Assume the hypotheses of the Surjective Lemma; define $\mathit{\Lambda}$ as in Lemma~\ref{lemma:sur-Lambda} with respect to $F:=Rf_1+\cdots+Rf_n$ and $X$; and let $\m\in \mathit{\Lambda}\cap\Max(R)$.  Then there exist elements $r_{\m,1},\ldots,r_{\m,n-1}\in R$ and a number $\mathscr{L}_{\m}\in\{1,\ldots,n\}$ with the following property:  For all $s_1,\ldots,s_n\in R-\m$ and for every $n\times n$ matrix $V$ with entries in $R$ such that 
\[
V
\equiv
\begin{pmatrix}
1 & 0 & \cdots & 0 & r_{\m,1} \\
0 & \ddots & \ddots & \vdots & \vdots \\
\vdots & \ddots & \ddots & 0 & \vdots \\
0 & \cdots & 0 & 1 & r_{\m,n-1} \\
0 & \cdots & \cdots & 0 & 1 \\
\end{pmatrix}
P_{\mathscr{L}_{\m}}
\begin{pmatrix}
s_1 & 0 & \cdots & 0 & 0 \\
0 & \ddots & \ddots & \vdots & \vdots \\
\vdots & \ddots & \ddots & 0 & \vdots \\
0 & \cdots & 0 & s_{n-1} & 0 \\
0 & \cdots & \cdots & 0 & s_n \\
\end{pmatrix}
\textnormal{ (mod }\m\textnormal{)},
\]
the first $n-1$ components of $Vf:=(g_1,\ldots,g_n)^{\top}$ form a map $(g_1,\ldots,g_{n-1})^{\top}$ that is $(t,X,\m)$-surjective.
\end{lemma}

\begin{proof}
For all $s_1,\ldots,s_n\in R-\m$, if $G:=Rs_1f_1+\cdots+Rs_nf_n$, then $\partial_{\m}(G)=\partial_{\m}(F)$. Hence, by Nakayama's Lemma, it suffices to prove the case in which $s_1=\cdots=s_n=1$.

Define every object in Definition~\ref{definition:sur-F} with respect to our current hypotheses, with $\m$ taking the place of $\p$ and with $A$ defined as the $n\times n$ identity matrix.  Then we have $f':=(f_1,\ldots,f_{n-1})^{\top}$.  Let $d:=\partial_{\m}(F)$, and let $B$ be a $d\times n$ matrix with entries in $R$ such that $(Bf)_{\m}$ is surjective.  Let $C\in\GL(d,R_{\m})$ such that $CB$ can be represented by a matrix $(b_{i,j})$ with entries in $R$ and such that $\overline{CB}$ is in the following reduced row echelon form, where the nonzero entries are clustered toward the top right corner of the matrix:
\[
\overline{CB}
=\begin{pmatrix}
\overline{0} & \cdots & \overline{0} & \overline{1} & \vdots & \overline{0} & \vdots & \overline{0} & \vdots \\
\overline{0} & \cdots & \overline{0} & \overline{0} & \vdots & \overline{1} & \vdots & \overline{0} & \vdots \\ 
\hdotsfor{9} \\
\overline{0} & \cdots & \overline{0} & \overline{0} & \vdots & \overline{0} & \vdots & \overline{1} & \vdots \\
\end{pmatrix}.
\]
Here, the vertical and horizontal ellipses denote possible omissions of entries, and the zero columns on the left may not be present.  For every $i\in\{1,\ldots,d\}$, let $j_i$ be the smallest number in the set $\{1,\ldots,n\}$ such that $\overline{b_{i,j_i}}\neq \overline{0}$.  We assume that, for every $i\in\{1,\ldots,d\}$, the entry $\overline{b_{i,j_i}}$ is the only nonzero entry in the $(j_i)$th column of $\overline{CB}$.

Below, we consider several cases and prove the lemma in each case.  Since $t$ and $X$ are understood, we may use the term \textit{$\m$-surjective} for the rest of the proof without the risk of ambiguity.

First suppose that $f'$ is $\m$-surjective.  Then, by Nakayama's Lemma, we may let $\mathscr{L}_{\m}=n$ and $r_{\m,1}=\cdots=r_{\m,n-1}=0$.

Suppose next that $j_d\leqslant n-1$.  Then, by Nakayama's Lemma, we may take $\mathscr{L}_{\m}=n$, and we may define $r_{\m,j}$ for every $j\in\{1,\ldots,n-1\}$ as follows: If $j=j_i$ for some $i\in\{1,\ldots,d\}$, then let $r_{\m,j}=b_{i,n}$; otherwise, let $r_{\m,j}=0$.

Now suppose, for the rest of the proof, that $f'$ is not $\m$-surjective and that $j_d=n$.  If $\partial_{\m}(F)=n$, then $\partial_{\m}(F')\geqslant \partial_{\m}(F)-1 = n-1$ by Lemma~\ref{lemma:sur-del-one}, and so $f'$ is $\m$-surjective, a contradiction.  Hence $\partial_{\m}(F)\leqslant n-1$.

Since $j_d=n$ and since $d=\partial_{\m}(F)\leqslant n-1$, there exists $k\in \{1,\ldots,n-1\}-\{j_1,\ldots,j_{d-1}\}$.  Accordingly, by Nakayama's Lemma, we may let $\mathscr{L}_{\m}$ be any such $k$, and we may define $r_{\m,j}$ for every $j\in\{1,\ldots,n-1\}$ as follows:  If $j=j_i$ for some $i\in\{1,\ldots,d-1\}$, then let $r_{\m,j}=b_{i,k}$; otherwise, let $r_{\m,j}=0$.
\end{proof}

We use the next lemma to find a matrix $Q\in\GL(n,R)$ and elements 
\[
s_1,\ldots,s_n\in R-\bigcup_{\m\in\mathit{\Lambda}\cap\Max(R)} \m
\]
such that
\[
Q
\equiv
P_{\mathscr{L}_{\m}}
\begin{pmatrix}
s_1 & 0 & \cdots & 0 & 0 \\
0 & \ddots & \ddots & \vdots & \vdots \\
\vdots & \ddots & \ddots & 0 & \vdots \\
0 & \cdots & 0 & s_{n-1} & 0 \\
0 & \cdots & \cdots & 0 & s_n \\
\end{pmatrix}
\textnormal{ (mod }\m\textnormal{)}
\]
for every $\m\in\mathit{\Lambda}\cap\Max(R)$.  We phrase the next lemma using language more general than this because we will use the lemma again under different circumstances in Sections~\ref{sec:sur-lemma} and~\ref{sec:spl} of this paper and in a separate paper on cancellation properties of modules~\cite{Bai2}.

\begin{lemma}\label{lemma:sur-Q}
Let $R$ be a commutative ring, and let $n\in\ZZ$ with $n\geqslant 2$.  Let $\mathit{\Lambda}_1,\ldots,\mathit{\Lambda}_n$ be finite, pairwise disjoint subsets of $\Max(R)$.  (Here, we allow some, or even all, of these sets to be empty.)  Then there exist a matrix $Q\in\textnormal{\textbf{GL}}(n,R)$ and elements
\[
s_1,\ldots,s_n\in R-\bigcup_{\m\in\mathit{\Lambda}_1\cup\cdots\cup\mathit{\Lambda}_n} \m
\]
such that, for every $i\in\{1,\ldots,n\}$ and for every $\m\in \mathit{\Lambda}_i$, the matrix $Q$ satisfies the congruence
\[
Q
\equiv
P_i
\begin{pmatrix}
s_1 & 0 & \cdots & 0 & 0 \\
0 & \ddots & \ddots & \vdots & \vdots \\
\vdots & \ddots & \ddots & 0 & \vdots \\
0 & \cdots & 0 & s_{n-1} & 0 \\
0 & \cdots & \cdots & 0 & s_n \\
\end{pmatrix}
\textnormal{ (mod }\m\textnormal{)}.
\]
Moreover, for any
\[
a\in R-\bigcup_{\m\in \mathit{\Lambda}_1} \m,
\]
we can arrange for the first row of $Q$ to be of the form
\[
\begin{pmatrix}
1-ab & 0 & \cdots & 0 & ab \\
\end{pmatrix}
\]
for some $b\in R$.
\end{lemma}

\begin{remark}
We do not use the last claim of Lemma~\ref{lemma:sur-Q} in this paper, but we use it in a paper on the cancellation of modules~\cite{Bai2}.
\end{remark}

\begin{proof}[Proof of Lemma~\ref{lemma:sur-Q}]
Let 
\[
a\in R-\bigcup_{\m\in \mathit{\Lambda}_1} \m.
\]
For every $i\in\{1,\ldots,n\}$, let
\[
I_i
:=\bigcap_{\m\in \mathit{\Lambda}_i} \m
\hspace{0.5in}\textnormal{and}\hspace{0.5in}
J_i:=\bigcap_{j\in \{1,\ldots,n\}-\{i\}} \left(\bigcap_{\m\in \mathit{\Lambda}_j} \m\right),
\]
and let
\[
U_i
:=\bigcup_{\m\in \mathit{\Lambda}_i} \m
\hspace{0.5in}\textnormal{and}\hspace{0.5in}
V_i:=\bigcup_{j\in \{1,\ldots,n\}-\{i\}}\left(\bigcup_{\m\in \mathit{\Lambda}_j} \m\right).
\]
We would like to prove that there exist
\[
\begin{array}{lllll}
a_1\in I_1-V_1, & a_2\in I_2-V_2, & \ldots, & a_{n-1}\in I_{n-1}-V_{n-1}, & \\
b_1\in aJ_1-U_1, & b_2\in J_2-U_2, & \ldots, & b_{n-1}\in J_{n-1}-U_{n-1}, & \\
c_1\in J_1-U_1, & c_2\in J_2-U_2, & \ldots, & c_{n-1}\in J_{n-1}-U_{n-1}, & c_n\in J_n-U_n \\
\end{array}
\]
such that $a_1=1-b_1$ and such that the $n\times n$ matrix
\[
Q
:=
\begin{pmatrix}
a_1 & 0 & \cdots & 0 & b_1 \\
0 & \ddots & \ddots & \vdots & \vdots \\
\vdots & \ddots & \ddots & 0 & \vdots \\
0 & \cdots & 0 & a_{n-1} & b_{n-1} \\
c_1 & \cdots & \cdots & c_{n-1} & c_n \\
\end{pmatrix}
\]
has determinant 1 and is thus invertible.  After we have accomplished this goal, we can appeal to the Chinese Remainder Theorem to produce
\[
s_1,\ldots,s_n\in R-\bigcup_{\m\in\mathit{\Lambda}_1\cup\cdots\cup\mathit{\Lambda}_n} \m
\]
such that the following conditions hold:
\begin{enumerate}
\item For every $i\in\{1,\ldots,n-1\}$ and for every $\m\in ( \mathit{\Lambda}_1\cup\cdots\cup \mathit{\Lambda}_n)- \mathit{\Lambda}_i$, the element $s_i$ satisfies the congruence $s_i\equiv a_i$ (mod $\m$).
\item For every $i\in\{1,\ldots,n-1\}$ and for every $\m\in \mathit{\Lambda}_i$, the element $s_n$ satisfies the congruence $s_n\equiv b_i$ (mod $\m$).
\item For every $i\in \{1,\ldots,n\}$ and for every $\m\in \mathit{\Lambda}_i$, the element $s_i$ satisfies the congruence $s_i\equiv c_i$ (mod $\m$).
\end{enumerate}
Finally, since $b_1\in aJ_1$, we can choose $b\in J_1$ such that $b_1=ab$.  Hence $a_1=1-b_1=1-ab$.  The matrix $Q$ and the elements $s_1,\ldots,s_n$ of $R$ will then jointly satisfy all of the conditions described in the lemma.

We note that, for any comaximal ideals $K,L$ of $R$ and for any $\alpha\in K$ and $\beta\in L$ such that $\alpha+\beta=1$, it is the case that $R=\sqrt{\alpha R+\beta R}=\sqrt{\alpha^2R+\beta^2R}\subseteq\sqrt{\alpha K+\beta L}$, and so $\alpha K+\beta L=R$.  We will use this observation shortly.

Next, we prove that $aJ_1+J_2+\cdots+J_n=R$.  Suppose not.  Then there exists $\n\in\Max(R)$ such that $aJ_1+J_2+\cdots+J_n\subseteq\n$, and so $aJ_1,J_2,\ldots,J_n\subseteq\n$.  Since $aJ_1\subseteq\n$, we see that $a\in\n$ or $J_1\subseteq\n$.  Either way, $\n\not\in\mathit{\Lambda}_1$.  On the other hand, $J_2,\ldots,J_n\subseteq\n$, and so $\n\in\mathit{\Lambda}_1$, a contradiction.  Hence $aJ_1+J_2+\cdots+J_n=R$, and so $J_1+J_2+\cdots+J_n=R$ as well.

Let $i\in\{1,\ldots,n-1\}$, and suppose that we have defined
\[
\begin{array}{llll}
a_1\in I_1-V_1, & a_2\in I_2-V_2, & \ldots, & a_{i-1}\in I_{i-1}-V_{i-1}, \\
b_1\in aJ_1-U_1, & b_2\in J_2-U_2, & \ldots, & b_{i-1}\in J_{i-1}-U_{i-1} \\
\end{array}
\]
and the ideals
\[
\begin{array}{rcl}
K_{i-1} & := &(b_1a_2\cdots a_{i-2}J_1)+\cdots+(a_1\cdots a_{i-3}b_{i-2}J_{i-2}) \\
& & +(a_1\cdots a_{i-2}J_i)+\cdots+(a_1\cdots a_{i-2}J_n) \\
\end{array}
\]
and
\[
L_{i-1}
:=a_1\cdots a_{i-2}J_{i-1}
\]
of $R$ so that $a_{i-1}K_{i-1}+b_{i-1}L_{i-1}=R$.  Let
\[
\begin{array}{rcl}
K_i & := &(b_1a_2\cdots a_{i-1}J_1)+\cdots+(a_1\cdots a_{i-2}b_{i-1}J_{i-1}) \\
& & +(a_1\cdots a_{i-1}J_{i+1})+\cdots+(a_1\cdots a_{i-1}J_n) \\
\end{array}
\]
and
\[
L_i
:=a_1\cdots a_{i-1}J_i
\]
so that $K_i+L_i=a_{i-1}K_{i-1}+b_{i-1}L_{i-1}=R$.  If $i=1$, then let $a_1\in K_1$ and $b_1\in aJ_1\subseteq L_1$ with $a_1+b_1=1$ so that $R=a_1K_1+b_1aJ_1\subseteq a_1K_1+b_1L_1$ and, hence, so that $a_1K_1+b_1L_1=R$.  If $i\geqslant 2$, then simply let $a_i\in K_i$ and $b_i\in L_i$ with $a_i+b_i=1$ so that $a_iK_i+b_iL_i=R$.

We will prove that $a_i\in I_i-V_i$ and that $b_i\in J_i-U_i$.  Certainly $a_i\in K_i\subseteq I_i$, and $b_i\in L_i\subseteq J_i$.  Suppose that $a_i\in V_i$.  Then there is $\n\in (\mathit{\Lambda}_1\cup\cdots\cup\mathit{\Lambda}_n)-\mathit{\Lambda}_i$ such that $a_i\in\n$.  Since $b_i\in J_i\subseteq \n$, we have $1=a_i+b_i\in\n$, a contradiction.  Hence $a_i\in I_i-V_i$.  Similarly, $b_i\in J_i-U_i$.

By induction on $i$, we can thus define 
\[
\begin{array}{llll}
a_1\in I_1-V_1, & a_2\in I_2-V_2, & \ldots, & a_{n-1}\in I_{n-1}-V_{n-1}, \\
b_1\in aJ_1-U_1, & b_2\in J_2-U_2, & \ldots, & b_{n-1}\in J_{n-1}-U_{n-1} \\
\end{array}
\]
and ideals
\[
K_{n-1} := (b_1a_2\cdots a_{n-2}J_1)+\cdots+(a_1\cdots a_{n-3}b_{n-2}J_{n-2})+(a_1\cdots a_{n-2}J_n)
\]
and
\[
L_{n-1}
:=a_1\cdots a_{n-2}J_{n-1}
\]
of $R$ so that $a_{n-1}K_{n-1}+b_{n-1}L_{n-1}=R$.  Hence
\[
(b_1a_2\cdots a_{n-1}J_1)
+\cdots
+(a_1\cdots a_{n-2}b_{n-1}J_{n-1})
+(a_1\cdots a_{n-1}J_n)
=R.
\]
Accordingly, we can choose
\[
\begin{array}{llll}
c_1\in J_1, & \ldots, & c_{n-1}\in J_{n-1}, & c_n\in J_n \\
\end{array}
\]
such that the determinant
\[
-(b_1a_2\cdots a_{n-1}c_1)
-\cdots
-(a_1\cdots a_{n-2}b_{n-1}c_{n-1})
+(a_1\cdots a_{n-1}c_n)
\]
of the matrix
\[
Q
:=\begin{pmatrix}
a_1 & 0 & \cdots & 0 & b_1 \\
0 & \ddots & \ddots & \vdots & \vdots \\
\vdots & \ddots & \ddots & 0 & \vdots \\
0 & \cdots & 0 & a_{n-1} & b_{n-1} \\
c_1 & \cdots & \cdots & c_{n-1} & c_n \\
\end{pmatrix}
\]
is equal to 1.

It remains to show that $c_i\not\in U_i$ for every $i\in \{1,\ldots,n\}$.  Suppose that $c_1\in U_1$.  Then there is $\n\in\mathit{\Lambda}_1$ such that $c_1\in\n$.  Since $J_2,\ldots,J_n\subseteq\n$, we have $c_2,\ldots,c_n\in\n$, and so $1=\det(Q)\in Rc_1+\cdots+Rc_n\subseteq\n$, a contradiction.  Hence $c_1\not\in U_1$.  Similarly, $c_i\not\in U_i$ for every $i\in \{2,\ldots,n\}$.
\end{proof}

We combine the results of the last two lemmas to achieve the goal of this section:

\begin{lemma}\label{lemma:sur-Max}
Assume the hypotheses of the Surjective Lemma, and define $\mathit{\Lambda}$ as in Lemma~\ref{lemma:sur-Lambda} with respect to $F:=Rf_1+\cdots+Rf_n$ and $X$.  Then there exists a matrix $V\in \textnormal{\textbf{GL}}(n,R)$ such that the first $n-1$ components of $Vf:=(g_1,\ldots,g_n)^{\top}$ form a map $(g_1,\ldots,g_{n-1})^{\top}$ that is $(t,X,\m)$-surjective for every $\m\in \mathit{\Lambda}\cap\Max(R)$.
\end{lemma}

\begin{proof}
For every $\m\in\mathit{\Lambda}\cap\Max(R)$, choose $r_{\m,1},\ldots,r_{\m,n-1}\in R$ and $\mathscr{L}_{\m}\in\{1,\ldots,n\}$ so that they jointly satisfy the conclusion of Lemma~\ref{lemma:sur-L}.  For every $i\in\{1,\ldots,n\}$, let
\[
\mathit{\Lambda}_i
=\{\m\in \mathit{\Lambda}\cap \Max(R):\mathscr{L}_{\m}=i\}.
\]
Then $\mathit{\Lambda}_1,\ldots,\mathit{\Lambda}_n$ are finite, pairwise disjoint subsets of $\Max(R)$.  Hence, by Lemma~\ref{lemma:sur-Q}, there exist a matrix $Q\in \GL(n,R)$ and elements
\[
s_1,\ldots,s_n\in R-\bigcup_{\m\in\mathit{\Lambda}_1\cup\cdots\cup\mathit{\Lambda}_n} \m
\]
such that, for every $i\in\{1,\ldots,n\}$ and for every $\m\in\mathit{\Lambda}_i$, we have
\[
Q
\equiv
P_i
\begin{pmatrix}
s_1 & 0 & \cdots & 0 & 0 \\
0 & \ddots & \ddots & \vdots & \vdots \\
\vdots & \ddots & \ddots & 0 & \vdots \\
0 & \cdots & 0 & s_{n-1} & 0 \\
0 & \cdots & \cdots & 0 & s_n \\
\end{pmatrix}
\textnormal{ (mod }\m\textnormal{)}.
\]
Next, we use the Chinese Remainder Theorem to find $r_1,\ldots,r_{n-1}\in R$ such that $r_i\equiv r_{\m,i}$ (mod $\m$) for every $i\in\{1,\ldots,n-1\}$ and for every $\m\in\mathit{\Lambda}\cap\Max(R)$, and we define
\[
U
:=\begin{pmatrix}
1 & 0 & \cdots & 0 & r_1 \\
0 & \ddots & \ddots & \vdots & \vdots \\
\vdots & \ddots & \ddots & 0 & \vdots \\
0 & \cdots & 0 & 1 & r_{n-1} \\
0 & \cdots & \cdots & 0 & 1 \\
\end{pmatrix}
\in \GL(n,R).
\]
Let $V:=UQ$.  Then 
\[
V
\equiv
\begin{pmatrix}
1 & 0 & \cdots & 0 & r_{\m,1} \\
0 & \ddots & \ddots & \vdots & \vdots \\
\vdots & \ddots & \ddots & 0 & \vdots \\
0 & \cdots & 0 & 1 & r_{\m,n-1} \\
0 & \cdots & \cdots & 0 & 1 \\
\end{pmatrix}
P_{\mathscr{L}_{\m}}
\begin{pmatrix}
s_1 & 0 & \cdots & 0 & 0 \\
0 & \ddots & \ddots & \vdots & \vdots \\
\vdots & \ddots & \ddots & 0 & \vdots \\
0 & \cdots & 0 & s_{n-1} & 0 \\
0 & \cdots & \cdots & 0 & s_n \\
\end{pmatrix}
\textnormal{ (mod }\m\textnormal{)}
\]
for every $\m\in\mathit{\Lambda}\cap\Max(R)$.  Now, Lemma~\ref{lemma:sur-L} tells us that the first $n-1$ components of $Vf:=(g_1,\ldots,g_n)^{\top}$ form a map $(g_1,\ldots,g_{n-1})^{\top}$ that is $(t,X,\m)$-surjective for every $\m\in\mathit{\Lambda}\cap\Max(R)$.
\end{proof}

In the next section, we complete our proof of the Surjective Lemma.

\section{A Proof of the Surjective Lemma}\label{sec:sur-lemma}

Throughout this section, we assume the hypotheses of the Surjective Lemma (Lemma~\ref{lemma:sur}), and we let $\mathit{\Lambda}$ be defined as in Lemma~\ref{lemma:sur-Lambda} with respect to $F:=Rf_1+\cdots+Rf_n$ and $X$.  Since $t$ and $X$ are understood, we may use the terms \textit{$\p$-surjective} and \textit{$Y$-surjective} for any $\p\in X$ and for any $Y\subseteq X$ without the risk of confusion.

In this section, we find a matrix $V\in \GL(n,R)$ such that the first $n-1$ components of $Vf:=(g_1,\ldots,g_n)^{\top}$ form a map $g:=(g_1,\ldots,g_{n-1})^{\top}$ that is $\mathit{\Lambda}$-surjective.  Lemma~\ref{lemma:sur-reduction} will then tell us that $g$ is $X$-surjective and, hence, that we have proved the Surjective Lemma.

\begin{proof}[Proof of the Surjective Lemma] Let $\q_1,\ldots,\q_m$ be the distinct members of $\mathit{\Lambda}-\Max(R)$, and arrange $\q_1,\ldots,\q_m$ so that, for every $\ell\in\{1,\ldots,m\}$, the prime $\q_{\ell}$ is a minimal member of the set $\{\q_1,\ldots,\q_{\ell}\}$.  We prove, by induction on $\ell\geqslant 0$, that there exists $V\in\GL(n,R)$ such that the first $n-1$ components of $Vf:=(g_1,\ldots,g_n)^{\top}$ form a map $(g_1,\ldots,g_{n-1})^{\top}$ that is $\p$-surjective for every $\p\in \mathit{\Lambda}-\{\q_{\ell+1},\ldots,\q_m\}$. 

Lemma~\ref{lemma:sur-Max} proves the case in which $\ell=0$.  Suppose then that $1\leqslant \ell\leqslant m$ and that there exists $A\in\GL(n,R)$ such that the first $n-1$ components of $f^*:=Af:=(f'_1,\ldots,f'_n)^{\top}$ form a map $f':=(f'_1,\ldots,f'_{n-1})^{\top}$ that is $\p$-surjective for every $\p\in \mathit{\Lambda}-\{\q_{\ell},\ldots,\q_m\}$.  If $f'$ happens to be $\q_{\ell}$-surjective as well, then we may set $V=A$ to finish the inductive step.  Suppose then that $f'$ is not $\q_{\ell}$-surjective.  Define every object in Definition~\ref{definition:sur-F} with respect to our current hypotheses, with $\q_{\ell}$ taking the place of $\p$.

Let
\[
J
:=\bigcap_{\p\in \mathit{\Lambda}-\{\q_{\ell},\ldots,\q_m\}} \p.
\]
It suffices to find $r_1,\ldots,r_{n-1}\in J$ such that, if
\[
U
:=\begin{pmatrix}
1 & 0 & \cdots & 0 & r_1 \\
0 & \ddots & \ddots & \vdots & \vdots \\
\vdots & \ddots & \ddots & 0 & \vdots \\
0 & \cdots & 0 & 1 & r_{n-1} \\
0 & \cdots & \cdots & 0 & 1 \\
\end{pmatrix}
\]
and if $Uf^*:=(g_1,\ldots,g_n)^{\top}$, then $G:=Rg_1+\cdots+Rg_{n-1}$ satisfies $\partial_{\q_{\ell}}(G)=\partial_{\q_{\ell}}(F)$:  Given such $r_1,\ldots,r_{n-1}\in J$, we will see that the first $n-1$ components of $Uf^*=UAf:=(g_1,\ldots,g_n)^{\top}$ form a map $(g_1,\ldots,g_{n-1})^{\top}$ that is not only $\q_{\ell}$-surjective but, by Nakayama's Lemma, also $\p$-surjective for every $\p\in\mathit{\Lambda}-\{\q_{\ell},\ldots,\q_m\}$.  Thus we will be able to take $V:=UA$ to finish the inductive step and, thus, the proof overall.  Before we find such $r_1,\ldots,r_{n-1}\in J$, though, we must complete some more preparatory work.

To simplify notation, let $\q:=\q_{\ell}$ from now on.  First we show that $\partial_{\q}(F')=\partial_{\q}(F)-1$ and that $\partial_{\q}(F)\leqslant n-1$.  By Lemma~\ref{lemma:sur-del-one} and by our assumption that $f'$ is not $\q$-surjective, we have
\[
\partial_{\q}(F)-1\leqslant \partial_{\q}(F')< \min\{n-1,t+\dim_X(\q)\}\leqslant \partial_{\q}(F),
\]
and so $\partial_{\q}(F')=\partial_{\q}(F)-1$.  Now, if $\partial_{\q}(F)=n$, then $\partial_{\q}(F')=\partial_{\q}(F)-1=n-1$, and so $f'$ is $\q$-surjective, a contradiction.  Hence $\partial_{\q}(F)\leqslant n-1$.

Let $d:=\partial_{\p}(F)$, and let $B$ be a $d\times n$ matrix with entries in $R$ such that $(Bf^*)_{\q}$ is surjective.  Let $C\in \GL(d,R_{\q})$ such that $CB$ can be represented by a matrix $(b_{i,j})$ with entries from $R$ and such that $\overline{CB}$ is in the following row echelon form with the nonzero entries clustered toward the top right corner of the matrix and with $s\in R-\q$:
\[
\overline{CB}
=\begin{pmatrix}
\overline{0} & \cdots & \overline{0} & \overline{s} & \vdots & \overline{0} & \vdots & \overline{0} & \vdots \\
\overline{0} & \cdots & \overline{0} & \overline{0} & \vdots & \overline{s} & \vdots & \overline{0} & \vdots \\ 
\hdotsfor{9} \\
\overline{0} & \cdots & \overline{0} & \overline{0} & \vdots & \overline{0} & \vdots & \overline{s} & \vdots \\
\end{pmatrix}.
\]
Here, the vertical and horizontal ellipses denote possible omissions of entries, and the zero columns on the left may not be present.  Now, for every $i\in\{1,\ldots,d\}$, let $j_i$ be the smallest number in the set $\{1,\ldots,n\}$ such that $\overline{b_{i,j_i}}\neq\overline{0}$.  We assume that, for every $i\in\{1,\ldots,d\}$, the entry $\overline{b_{i,j_i}}$ is the only nonzero entry in the $(j_i)$th column of $\overline{CB}$.  Let 
\[
B^*
:=(b^*_{i,j})
:=\begin{pmatrix}
0 & \cdots & 0 & s & \vdots & 0 & \vdots & 0 & \vdots \\
0 & \cdots & 0 & 0 & \vdots & s & \vdots & 0 & \vdots \\ 
\hdotsfor{9} \\
0 & \cdots & 0 & 0 & \vdots & 0 & \vdots & s & \vdots \\
\end{pmatrix}
\]
be a $d\times n$ matrix with entries in $R$ that satisfies the following conditions:
\begin{enumerate}
\item $\overline{B^*}=\overline{CB}$.
\item For every $i\in\{1,\ldots,d\}$ and for every $j\in\{1,\ldots,n\}$, if $\overline{b_{i,j}}=\overline{0}$, then $b^*_{i,j}=0$.
\item For every $i\in\{1,\ldots,d\}$, we have $b^*_{i,j_i}=s$.
\end{enumerate}
Hence $B^*f^*\in F^{\oplus d}$, and $\overline{B^*f^*}$ is surjective.  Nakayama's Lemma then tells us that $(B^*f^*)_{\q}$ is surjective.  Thus, we assume, without loss of generality, that $B=CB=(b_{i,j})$ and that $B$ already has the desirable form of $B^*$. 

Since $(Bf^*)_{\q}$ is surjective, there exists a finitely generated $R$-submodule $L$ of $M$ such that the restriction of $(Bf^*)_{\q}$ to $L_{\q}$ is surjective.  We may assume, then, without loss of generality, that $M_{\q}$ is a finitely generated $R_{\q}$-module.

Let $\mu:=\mu_{R_{\q}}(M_{\q})$, and let $\nu:=\mu_{R_{\q}}(N_{\q})$.  Since $\q\in\Supp_R(N)$, we see that $\nu\geqslant 1$.  Since $f$ is $\q$-surjective, $d\geqslant t\geqslant 1$, and so Nakayama's Lemma tells us that $\mu\geqslant d\nu\geqslant 1$.  In fact, without loss of generality, we may assume that $\mu=d\nu$.

Let $E:=(\varepsilon_1,\ldots,\varepsilon_{ d\nu})^{\top}$ be an ordered $ d\nu$-tuple of elements of $M_{\q}$ such that $\{\varepsilon_1,\ldots,\varepsilon_{ d\nu}\}$ is a minimal generating set for $M_{\q}$ over $R_{\q}$, and let $Z:=(\zeta_1,\ldots,\zeta_{\nu})^{\top}$ be an ordered $\nu$-tuple of elements of $N_{\q}$ such that $\{\zeta_1,\ldots,\zeta_{\nu}\}$ is a minimal generating set for $N_{\q}$ over $R_{\q}$.  For every $i\in \{1,\ldots,n\}$, let $\varphi'_i:=(f'_i)_{\q}$, and let $\mathit{\Phi}'_i$ be a $\nu\times  d\nu$ matrix with entries in $R_{\q}$ that represents $\varphi'_i$ with respect to $E$ and $Z$ in the following sense:  For every $j\in\{1,\ldots, d\nu\}$, if $\theta_{1,j},\ldots,\theta_{\nu,j}\in R_{\q}$ such that $\varphi'_i(\varepsilon_j)=\theta_{1,j}\zeta_1+\cdots+\theta_{\nu,j}\zeta_{\nu}$, then we may define the $j$th column of $\mathit{\Phi}'_i$ to be
\[
\begin{pmatrix}
\theta_{1,j} \\
\vdots \\
\theta_{\nu,j} \\
\end{pmatrix}.
\] Now let $\mathit{\Phi}^*$ be the $n\nu \times  d\nu$ matrix whose $i$th $\nu\times  d\nu$ block is $\mathit{\Phi}'_i$.  Hence
\[
\mathit{\Phi}^*=
\begin{pmatrix}
\mathit{\Phi}'_1 \\
\vdots \\
\mathit{\Phi}'_n \\
\end{pmatrix}.
\]

Finally, we let $\rank(\mathit{\Xi})$ denote the rank of a matrix $\mathit{\Xi}$ with entries in $\kappa(\q)$.

We now return to the task of finding $r_1,\ldots,r_{n-1}\in J$ that satisfy the criteria described earlier.  We consider two cases.\\

\setlength{\leftskip}{0.25in}

\underline{Case 1:  $j_d\leqslant n-1$.}\vspace{1mm}
In this case, $B$ has the following form:  
\[
B
=\begin{pmatrix}
\vdots & s & \vdots & 0 & \vdots & 0 & \vdots & b_{1,n} \\
\vdots & 0 	& \vdots & s & \vdots & 0 & \vdots & b_{2,n} \\ 
\hdotsfor{8} \\
\vdots & 0 & \vdots & 0 & \vdots & s & \vdots & b_{d,n} \\
\end{pmatrix}.
\]
Let $r_j=0\in J$ for every $j\in \{1,\ldots,n-1\}-\{j_1,\ldots,j_d\}$.

Let $i\in\{1,\ldots,d\}$, and suppose that we have defined $r_{j_1},\ldots,r_{j_{(i-1)}}\in J$.  Let
\[
B_i
:=\begin{pmatrix}
\vdots & s	& \vdots & 0 & \vdots & 0 & \vdots & 0 & \vdots & sr_{j_1} \\
\hdotsfor{10} \\
\vdots & 0 & \vdots & s & \vdots & 0 & \vdots & 0 & \vdots & sr_{j_{(i-1)}} \\ 
\vdots & 0 & \vdots & 0 & \vdots & s & \vdots & 0 & \vdots & 0 \\
\vdots & 0 	& \vdots & 0 & \vdots & 0 & \vdots & s & \vdots& b_{i+1,n} \\
\hdotsfor{10} \\
\end{pmatrix}
\]
be the $d\times n$ matrix obtained from $B$ by replacing $b_{1,n},\ldots,b_{i-1,n},b_{i,n}$ with $sr_{j_1},\ldots,sr_{j_{(i-1)}},$ $0$, respectively.  Let $\mathit{\Omega}_i:=(B_i\otimes I_{\nu})\mathit{\Phi}^*$, and let
\[
\mathit{\Omega}'_i
:=\begin{pmatrix}
0 \\
\vdots \\
0 \\
\mathit{\Phi}'_n \\
0 \\
\vdots \\
0 \\
\end{pmatrix}
\]
be the $d\nu\times d\nu$ matrix obtained by replacing the $i$th $\nu\times d\nu$ block of the zero $d\nu\times d\nu$ matrix with $\mathit{\Phi}'_n$.  Suppose that $\rank\left(\overline{\mathit{\Omega}_i+b_{i,n}\mathit{\Omega}'_i}\right)=d\nu$.  We will prove that there exists $r_{j_i}\in J$ such that $\rank\left(\overline{\mathit{\Omega}_i+sr_{j_i}\mathit{\Omega}'_i}\right)=d\nu$.  

Let $\mathscr{I}$ denote the ideal $(sJ+\q)/\q$ of $R/\q$.  Since $\q$ is a nonmaximal prime ideal of $R$, we see that $R/\q$ is an infinite domain.  Since $s\in R-\q$ and since $J\not\subseteq\q$, the ideal $\mathscr{I}$ is nonzero, hence infinite.

Let
\[
\mathscr{S}_i:=\left\{\sigma\in\kappa(\q):\rank\left(\overline{\mathit{\Omega}_i}+\sigma\overline{\mathit{\Omega}'_i}\right)\leqslant d\nu-1\right\}.
\]
We will show that $\mathscr{I}$ contains an element $\rho_i$ that avoids $\mathscr{S}_i$.   Let $\mathscr{D}_i(x)$ denote the determinant of $\overline{\mathit{\Omega}_i}+x\overline{\mathit{\Omega}'_i}$, where $x$ is a variable.  Since $\rank\left(\overline{\mathit{\Omega}_i+b_{i,n}\mathit{\Omega}'_i}\right)=d\nu$, we see that $\mathscr{D}_i\left(\overline{b_{i,n}}\right)\neq\overline{0}$.  Hence $\mathscr{D}_i(x)$ is a nonzero polynomial.  Since the degree of $\mathscr{D}_i(x)$ is at most $\nu$, we see that $|\mathscr{S}_i|\leqslant \nu$.  Since $\mathscr{I}$ is infinite, $\mathscr{I}$ must then contain an element $\rho_i$ that avoids $\mathscr{S}_i$.

Now let $r_{j_i}\in J$ such that $\overline{sr_{j_i}}=\rho_i$.  Then $\rank\left(\overline{\mathit{\Omega}_i+sr_{j_i}\mathit{\Omega}'_i}\right)=d\nu$, as promised.

By induction, then, we can define matrices $B_1,\mathit{\Omega}_1,\mathit{\Omega}'_1,\ldots,B_d,\mathit{\Omega}_d,\mathit{\Omega}'_d$ and $r_{j_1},\ldots,r_{j_d}\in J$ such that $\rank\left(\overline{\mathit{\Omega}_d+sr_{j_d}\mathit{\Omega}'_d}\right)=d\nu$.  Now, let $B'$ be the $d\times (n-1)$ matrix obtained by deleting the $n$th column of $B$; let
\[
U
:=\begin{pmatrix}
1 & 0 & \cdots & 0 & r_1 \\
0 & \ddots & \ddots & \vdots & \vdots \\
\vdots & \ddots & \ddots & 0 & \vdots \\
0 & \cdots & 0 & 1 & r_{n-1} \\
0 & \cdots & \cdots & 0 & 1 \\
\end{pmatrix};
\]
and let $\mathit{\Gamma}$ denote the $(n-1)\nu\times  d\nu$ matrix obtained by deleting the $n$th $\nu\times  d\nu$ block of $(U\otimes I_{\nu})\mathit{\Phi}^*$.  Then $(B'\otimes I_{\nu})\mathit{\Gamma}=\mathit{\Omega}_d+sr_{j_d}\mathit{\Omega}'_d$, and so $\rank\left[\overline{(B'\otimes I_{\nu})\mathit{\Gamma}}\right]=d\nu$.  Let $Uf^*:=(g_1,\ldots,g_n)^{\top}$, and let $G:=Rg_1+\cdots+Rg_{n-1}$.  Then, by Nakayama's Lemma, $(B'\otimes I_{\nu})\mathit{\Gamma}$ represents a surjection in $G^{\oplus d}_{\q}$ from $M_{\q}$ to $N_{\q}^{\oplus d}$, and so $\partial_{\q}(G)=d=\partial_{\q}(F)$, as desired. \\

\underline{Case 2:  $j_d=n$.}\vspace{1mm}
In this case, $B$ has the following form:
\[
B
=\begin{pmatrix}
\vdots & s & \vdots & 0 & \vdots & 0 \\
\vdots & 0 	& \vdots & s & \vdots & 0 \\ 
\hdotsfor{6} \\
\vdots & 0 & \vdots & 0 & \vdots & s \\
\end{pmatrix}.
\]
Since $d\leqslant n-1$ and since $j_d=n$, there is $k\in \{1,\ldots,n-1\}-\{j_1,\ldots,j_{d-1}\}$.
Let $\mathit{\Omega}:=(B\otimes I_{\nu})\mathit{\Phi}^*$, and let
\[
\mathit{\Omega}'
:=\begin{pmatrix}
0 \\
\vdots \\
0 \\
\mathit{\Phi}'_k \\
\end{pmatrix}
\]
be the $d\nu\times  d\nu$ matrix obtained by replacing the $d$th $\nu\times d\nu$ block of the zero $d\nu\times d\nu$ matrix with $\mathit{\Phi}'_k$.

Let $\mathscr{J}$ denote the ideal $(J+\q)/\q$ of $R/\q$.  Since $\mathscr{I}\subseteq\mathscr{J}$, we see that $\mathscr{J}$ is infinite.  

Let
\[
\mathscr{S}:=\left\{\sigma\in\kappa(\q):\rank\left(\overline{\mathit{\Omega}}+\sigma\overline{\mathit{\Omega}'}\right)\leqslant d\nu-1\right\}.
\]
We will show that $\mathscr{J}$ contains a nonzero element $\rho$ such that $\rho^{-1}$ avoids $\mathscr{S}$.  Let $\mathscr{D}(x)$ be the determinant of $\overline{\mathit{\Omega}}+x\overline{\mathit{\Omega}'}$, where $x$ is a variable.  Since $\rank\left(\overline{\mathit{\Omega}+0\mathit{\Omega}'}\right)=\rank\left(\overline{\mathit{\Omega}}\right)=d\nu$, we see that $\mathscr{D}\left(\overline{0}\right)\neq\overline{0}$.  Hence $\mathscr{D}(x)$ is a nonzero polynomial.  Since the degree of $\mathscr{D}(x)$ is at most $\nu$, we see that $|\mathscr{S}|\leqslant \nu$.  Since $\mathscr{J}$ is infinite, $\mathscr{J}$ must then contain a nonzero element $\rho$ such that $\rho^{-1}$ avoids $\mathscr{S}$.

Now let $r\in J-\q$ such that $\overline{r}=\rho$, and let $1/r$ denote the multiplicative inverse of the element $r/1$ of $R_{\q}$ so that $\overline{(1/r)}=\rho^{-1}$.

Let
\[
B_1=\begin{pmatrix}
\vdots & b_{1,k} & \vdots & 0 \\
\vdots & b_{2,k} & \vdots & 0 \\
\hdotsfor{4} \\
\vdots & b_{d-1,k} & \vdots & 0 \\ 
\vdots & 1/r & \vdots & s \\
\end{pmatrix}
\]
be the $d\times n$ matrix obtained from $B$ by replacing $b_{d,k}=0$ with $1/r$.

Note that $(B_1\otimes I_{\nu})\mathit{\Phi}^*=\mathit{\Omega}+(1/r)\mathit{\Omega}'$ so that $\rank\left[\overline{(B_1\otimes I_{\nu})\mathit{\Phi}^*}\right]=d\nu$.

Next, let
\[
B_2:=\begin{pmatrix}
1 & 0 & \cdots & 0 & -rb_{1,k} \\
0 & \ddots & \ddots & \vdots & -rb_{2,k} \\
\vdots & \ddots & \ddots & 0 & \vdots \\
0 & \cdots & 0 & 1 & -rb_{d-1,k} \\
0 & \cdots & \cdots & 0 & rs \\
\end{pmatrix}
\in\GL(d,R_{\q}).
\]
Then $\rank\left[\overline{(B_2B_1\otimes I_{\nu})\mathit{\Phi}^*}\right]=d\nu$.  Also,
\[
B_2B_1
=\begin{pmatrix}
\vdots & 0 & \vdots & s(-rb_{1,k}) \\
\vdots & 0 & \vdots & s(-rb_{2,k}) \\
\hdotsfor{4} \\
\vdots & 0 & \vdots & s(-rb_{d-1,k}) \\ 
\vdots & s & \vdots & s(rs) \\
\end{pmatrix},
\]
where the column $(0,0,\ldots,0,s)^{\top}$ displayed above is the $k$th column of $B_2B_1$.  Now permute the rows of $B_2B_1$ to yield a matrix $B_3$ such that $\overline{B_3}$ is in row echelon form.  Then $\rank\left[\overline{(B_3\otimes I_{\nu})\mathit{\Phi}^*}\right]=d\nu$, and so we have reduced to Case 1.\\

\setlength{\leftskip}{0in}

This completes the inductive step of our proof.
\end{proof}

Now that we have worked through our proof of the Surjective Lemma, we can reveal why we address the members of $\mathit{\Lambda}\cap\Max(R)$ separately in Section~\ref{sec:Max}.  Let $\q$ be defined as in the main inductive step of this section.  Since $\q$ is a nonmaximal prime ideal of $R$, we see that $R/\q$ is an infinite domain and, hence, that the nonzero ideals $\mathscr{I}$ and $\mathscr{J}$ of $R/\q$ are also infinite.  In Case 1, we find that, for every $i\in\{1,\ldots,d\}$, there is an element of $\mathscr{I}$ that avoids $\mathscr{S}_i$ since $|\mathscr{S}_i|\leqslant \mu_{R_{\q}}(N_{\q})<\infty$.  In Case 2, we find that $\mathscr{J}$ must have a nonzero element whose multiplicative inverse avoids $\mathscr{S}$ since $|\{\overline{0}\}\cup\mathscr{S}|\leqslant 1+\mu_{R_{\q}}(N_{\q})<\infty$.  These lines of reasoning, \textit{mutatis mutandis}, are not necessarily available for a given $\m\in\mathit{\Lambda}\cap\Max(R)$.  In particular, it is not always the case that $R/\m$ is infinite or even that $|R/\m|\geqslant 2+\mu_{R_{\m}}(N_{\m})$.  Hence, in general, we cannot mimic the method that we use on the members of $\mathit{\Lambda}-\Max(R)$ to treat the members of $\mathit{\Lambda}\cap\Max(R)$.  The possibility that $|R/\m|\leqslant 1+\mu_{R_{\m}}(N_{\m})$ for some $\m\in\mathit{\Lambda}\cap\Max(R)$ is what compelled us to find a special method for dealing with the members of $\mathit{\Lambda}\cap\Max(R)$, and it is this method that we present in Section~\ref{sec:Max}.

This is not the only method that works.  In fact, there is an alternative to the method of Sections~\ref{sec:Max} and~\ref{sec:sur-lemma} that automatically simplifies our proof of the Surjective Lemma in a special case:  With respect to Lemma~\ref{lemma:sur-L}, we can define
\[
\mathit{\Lambda}_i:=\{\m\in\mathit{\Lambda}\cap\Max(R):\mathscr{L}_{\m}=i\textnormal{ and }|R/\m|\leqslant 1+\mu_{R_{\m}}(N_{\m})\}
\]
for every $i\in\{1,\ldots,n\}$, use Lemmas~\ref{lemma:sur-L} and~\ref{lemma:sur-Q} to account for the members of $\mathit{\Lambda}_1\cup\cdots\cup\mathit{\Lambda}_n$ only, and then proceed by induction on the remaining members of $\mathit{\Lambda}$ as in Section~\ref{sec:sur-lemma}.  The benefit of this approach is that, if $|R/\m|\geqslant 2+\mu_{R_{\m}}(N_{\m})$ for every $\m\in\mathit{\Lambda}\cap\Max(R)$, then $\mathit{\Lambda}_1,\ldots,\mathit{\Lambda}_n$ are all empty, and so it suffices to use the method of Section~\ref{sec:sur-lemma} for the entirety of~$\mathit{\Lambda}$.  

In another special case, the proof of the Surjective Lemma does not go through as quickly, but the result of the inductive step is still comparable to that of Section~\ref{sec:sur-lemma}: If $\mu_{R_{\m}}(N_{\m})=1$ for every $\m\in\mathit{\Lambda}\cap\Max(R)$, then we can mimic the method used in the inductive step of the proof of~\cite[Lemma~3.8]{DSPY} for the members of $\mathit{\Lambda}\cap\Max(R)$, and we can apply the method of Section 5 of this paper for the remaining members of $\mathit{\Lambda}$.  This case is noteworthy in the following sense:  Let $\q_1,\ldots,\q_m$ be the distinct members of $\mathit{\Lambda}$, and arrange $\q_1,\ldots,\q_m$ so that, for every $\ell\in\{1,\ldots,m\}$, the prime $\q_{\ell}$ is a minimal member of the set $\{\q_1,\ldots,\q_{\ell}\}$.  Also, list the members of $\mathit{\Lambda}\cap\Max(R)$ first.  Let $\ell\in\{1,\ldots,m\}$, and let 
\[
J:=\bigcap_{i=1}^{\ell-1} \q_i.
\]
Suppose that there exists $A\in\GL(n,R)$ such that the first $n-1$ components of $f^*:=Af:=(f'_1,\ldots,f'_n)^{\top}$ form a map $(f'_1,\ldots,f'_{n-1})^{\top}$ that is $\q_i$-surjective for every $i\in\{1,\ldots,\ell-1\}$.  Then there exist $r_1,\ldots,r_{n-1}\in J$ such that, if
\[
U
:=\begin{pmatrix}
1 & 0 & \cdots & 0 & r_1 \\
0 & \ddots & \ddots & \vdots & \vdots \\
\vdots & \ddots & \ddots & 0 & \vdots \\
0 & \cdots & 0 & 1 & r_{n-1} \\
0 & \cdots & \cdots & 0 & 1 \\
\end{pmatrix},
\]
then the first $n-1$ components of $Uf^*:=(g_1,\ldots,g_n)^{\top}$ form a map $(g_1,\ldots,g_{n-1})^{\top}$ that is $\q_i$-surjective for every $i\in\{1,\ldots,\ell\}$.

We can combine the two special cases that we have mentioned to yield the following corollary of the Surjective Lemma:

\begin{corollary}\label{corollary:sur-res}
Assume the hypotheses of the Surjective Lemma.  Define $\mathit{\Lambda}$ as in Lemma~\ref{lemma:sur-Lambda} with respect to $F:=Rf_1+\cdots+Rf_n$ and $X$.  Suppose that, for every $\m\in\mathit{\Lambda}\cap\Max(R)$, one of the following conditions holds:
\begin{enumerate}
\item $|R/\m|\geqslant 2+\mu_{R_{\m}}(N_{\m})$.
\item $\mu_{R_{\m}}(N_{\m})=1$.
\end{enumerate}
(For example, we may suppose that every residue field of $R$ is infinite or that $N$ is a locally cyclic $R$-module.)  Then there exist $r_1,\ldots,r_{n-1}\in R$ such that $(f_1+r_1f_n,\ldots,f_{n-1}+r_{n-1}f_n)^{\top}$ is $(t,X,X)$-surjective.
\end{corollary}

The conclusion of this corollary can be compared to that of~\cite[Theorem~B]{EE}.  We leave it to the reader to spell out the ramifications of this corollary for Theorems~\ref{theorem:sur-Bas},~\ref{theorem:sur-Ser},~\ref{theorem:sur-prelim}, and~\ref{theorem:sur-Bas-var}.

Despite their benefits, the alternative approaches to the Surjective Lemma have an obvious drawback:  Conditions on the sizes of residue fields and minimal generating sets do not receive proper context until the middle of Section~\ref{sec:sur-lemma}.  For this reason, we decided to present a method that avoids specific reference to the sizes of residue fields and minimal generating sets when handling the members of $\mathit{\Lambda}\cap\Max(R)$.  We remain faithful to this decision in our presentation of the analogous results of the next section.

\section{Proofs of Theorems~\texorpdfstring{\ref{theorem:spl-Bas}}{0.5} and~\texorpdfstring{\ref{theorem:spl-Ser}}{0.9}}\label{sec:spl}

Throughout this section, let $R$ denote a commutative ring; let $S$ denote a module-finite $R$-algebra; let $M$ denote a right $S$-module; and let $N$ denote a finitely presented right $S$-module.  As before, we view every left and right $S$-module as a standard $R$-module in the natural way.

In this section, we prove Theorems~\ref{theorem:spl-Bas} and~\ref{theorem:spl-Ser}.  Since many of the techniques here are similar to those that we use in Sections~\ref{sec:sur-Ser}--\ref{sec:sur-lemma}, we do not provide as much detail here as before.  Still, we state all of the necessary definitions and lemmas, and we indicate the major differences between the proofs here and their earlier analogues.

Of note is the fact that there are only four results in this section that do not require $N$ to be finitely presented over $S$:  For Remark~\ref{remark:spl-infty}, Remark~\ref{remark:spl-nXp}, Lemma~\ref{lemma:spl-del-one}, and Lemma \ref{lemma:spl-L}, it suffices for $N$ to be finitely generated over $S$.  Every other result in this section ultimately relies on Lemma~\ref{lemma:spl-closed}, and Lemma~\ref{lemma:spl-closed} relies on the finite presentation of $N$ over $S$.

We begin with the following definitions and remarks:

\begin{definition}\label{definition:delta}
Let $F$ be an $R$-submodule of $\Hom_S(M,N)$, and let $\p\in\Spec(R)$.  We let $\delta(F)$ denote the supremum of the nonnegative integers $t$ such that there exists $f\in F^{\oplus t}\subseteq \Hom_S(M,N^{\oplus t})$ that is split surjective over $S$.  We let $\delta_{\p}(F)$ denote the supremum of the nonnegative integers $t$ such that there exists $f\in F^{\oplus t}$ with the property that $f_{\p}$ is split surjective over $S_{\p}$.
\end{definition}

\begin{remark}\label{remark:spl-infty}
Let $F$ be a finitely generated $R$-submodule of $\Hom_S(M,N)$.  We observe that $\delta(F)=\infty$ if and only if $N=0$:  Certainly, if $N=0$, then $\delta(F)=\infty$.  On the other hand, if $\delta(F)=\infty$, then $\partial(F)=\infty$, and so $N=0$ by Remark~\ref{remark:infty}.

Let $\p\in\Spec(R)$.  Then, by the preceding discussion, $\delta_{\p}(F)=\infty$ if and only if $\p\not\in\Supp_R(N)$.
\end{remark}

\begin{definition}\label{definition:spl-tXp}
Let $n,t$ be positive integers with $n\geqslant t$; let $\p\in X\subseteq\Spec(R)$; and let $f:=(f_1,\ldots,f_n)^{\top}\in\Hom_S(M,N^{\oplus n})$.  We say that $f$ is \textit{$(t,X,\p)$-split} if $\delta_{\p}(Rf_1+\cdots + Rf_n)\geqslant \min\{n,t+\dim_X(\p)\}$.

Let $Y\subseteq X$.  We say that $f$ is \textit{$(t,X,Y)$-split} if $f$ is $(t,X,\q)$-split for every $\q\in Y$.

When $t$ and $X$ are understood, we use the terms \textit{$\p$-split} and \textit{$Y$-split} in place of \textit{$(t,X,\p)$-split} and \textit{$(t,X,Y)$-split}, respectively. 
\end{definition}

\begin{remark}\label{remark:spl-nXp}
Maintaining the hypotheses in the previous definition, we see that $f$ is $(n,X,\p)$-split if and only if $f_{\p}$ is split surjective over $S_{\p}$.  The reasoning is basically the same as in Remark~\ref{remark:nXp}.
\end{remark}

We now state an analogue of the Surjective Lemma (Lemma~\ref{lemma:sur}).

\begin{lemma}[Splitting Lemma]\label{lemma:spl}
Let $n,t$ be positive integers with $n\geqslant 1+t$, and let $X$ be a subset of $\Supp_R(N)$ that is a basic set for $R$.  Let $f:=(f_1,\ldots,f_n)^{\top}\in\Hom_S(M,N^{\oplus n})$, and suppose that $f$ is $(t,X,X)$-split.  Then there exist $f'_1,\ldots,f'_{n-1}\in Rf_1+\cdots + Rf_n$ such that $f':=(f'_1,\ldots, f'_{n-1})^{\top}$ is $(t,X,X)$-split.
\end{lemma}

We defer the proof of the Splitting Lemma to the end of this section.  Assuming the truth of the Splitting Lemma, we could prove the following theorem at this point, but we omit the proof since it is basically the same as the proof of Theorem~\ref{theorem:sur-prelim}.  Still, we would like to make one note about the proof.  We use the finite presentation of $N$ over $S$ more than just through the use of Lemma~\ref{lemma:spl-closed}:  When applying Part~(1) of Theorem~\ref{theorem:spl-prelim} to prove Part~(2), we use the fact that a map $g\in\Hom_S(M,N^{\oplus t})$ is split surjective over $S$ if and only if $g_{\m}$ is split surjective over $S_{\m}$ for every $\m\in\Max(R)\cap\Supp_R(N)$.

\begin{theorem}\label{theorem:spl-prelim}
Let $L$ be an $S$-submodule of $M$; let $F$ be a finitely generated $R$-submodule of $\Hom_S(L,N)$; and let $G$ be an $R$-submodule of $\Hom_S(M,N)$.  Suppose that every member of $F$ can be extended to a member of $G$.  Let $X$ be a subset of $\Supp_R(N)$ that is a basic set for $R$, and suppose that $\dim(X)<\infty$.  Then the following statements hold:
\begin{enumerate}
\item Let $t$ be a positive integer, and suppose that $\delta_{\p}(F)\geqslant t+\dim_X(\p)$ for every $\p\in X$.  Then there exists $g\in G^{\oplus t}$ such that $g_{\p}$ is split surjective over $S_{\p}$ for every $\p\in X$.
\item Suppose that $\Max(R)\cap\Supp_R(N)\subseteq X$.  Then
\[
\delta(G)\geqslant\inf\{\delta_{\p}(F)-\dim_X(\p):\p\in X\}.
\]
\end{enumerate}
\end{theorem}

We could prove Theorem~\ref{theorem:spl-Ser} at this point, but we omit the proof since it is similar to the proof of Theorem~\ref{theorem:sur-Ser}.

As with the Surjective Lemma, we can reduce the proof of the Splitting Lemma to a consideration of a finite subset $\mathit{\Lambda}$ of $X$.  The following lemma, which is analogous to Lemma~\ref{lemma:sur-closed}, helps us reach this goal.  As we mention above, this lemma marks the main point in this section that relies on the finite presentation of $N$ over $S$.

\begin{lemma}\label{lemma:spl-closed}
Let $F$ be an $R$-submodule of $\Hom_S(M,N)$, and let $t$ be a nonnegative integer.  Then the set $\{\p\in\Spec(R):\delta_{\p}(F)>t\}$ is open, and so the set $\{\p\in\Spec(R):\delta_{\p}(F)\leqslant t\}$ is closed.  Hence, for every subspace $X$ of $\Spec(R)$, the set $Y_t:=\{\p\in X:\delta_{\p}(F)\leqslant t\}$ is closed in $X$.
\end{lemma}

\begin{proof}
Let $\p\in\Spec(R)$ such that $\delta_{\p}(F)>t$; let $f\in F^{\oplus (t+1)}$ such that $f_{\p}$ is split surjective over $S_{\p}$; and let $L=N^{\oplus (t+1)}$.  Since $N$ is finitely presented over $S$, there exist $g\in\Hom_S(L,M)$ and $s\in R-\p$ such that $f\circ g=s\cdot 1_L$.  Let $U:=\{\q\in\Spec(R):s\not\in\q\}$.  Then $f_{\q}$ is split surjective over $S_{\q}$ for every $\q\in U$. Hence $U$ is an open neighborhood of $\p$ such that $\delta_{\q}(F)>t$ for every $\q\in U$.  Thus the set $\{\p\in\Spec(R):\delta_{\p}(F)>t\}$ is open.  This proves the first claim of the lemma.  The last two claims of the lemma follow from the first claim.
\end{proof}

We state the following lemma with an eye toward Theorem~\ref{theorem:spl-Bas}.  We omit the proof of this lemma since it is similar to the proof of Lemma~\ref{lemma:sur-fg}.

\begin{lemma}\label{lemma:spl-fg}
Suppose that $M$ is a direct summand of a direct sum of finitely presented right $S$-modules.  Let $X$ be a Noetherian subspace of $\Supp_R(N)$, and suppose that $\dim(X)<\infty$.  Let $t$ be a positive integer, and suppose that $\spl_{S_{\p}}(M_{\p},N_{\p})\geqslant t+\dim(X)$ for every $\p\in X$.  Then there exists a finitely generated $R$-submodule $F$ of $\Hom_S(M,N)$ such that $\delta_{\p}(F)\geqslant t+\dim(X)$ for every $\p\in X$.
\end{lemma}

We would now be in a position to prove Theorem~\ref{theorem:spl-Bas}, modulo the Splitting Lemma.  We omit the proof since it is similar to the proof of Theorem~\ref{theorem:sur-Bas}.

We would now also be able to prove the following variations of Lemma~\ref{lemma:spl-fg} and Theorems~\ref{theorem:spl-prelim} and~\ref{theorem:spl-Bas}, but we omit the proofs.  As with Lemma~\ref{lemma:sur-fg-var} and Theorem~\ref{theorem:sur-Bas-var}, the following variations are noteworthy in the sense that they do not require $M$ to be a direct summand of a direct sum of finitely presented right $S$-modules.

\begin{lemma}\label{lemma:spl-fg-var}
Let $F$ be an $R$-submodule of $\Hom_S(M,N)$.  Let $X$ be a Noetherian subspace of $\Supp_R(N)$, and suppose that $\dim(X)<\infty$.  Let $t$ be a positive integer, and suppose that $\delta_{\p}(F)\geqslant t+\dim(X)$ for every $\p\in X$.  Then there exists a finitely generated $R$-submodule $F'$ of $F$ such that $\delta_{\p}(F')\geqslant t+\dim(X)$ for every $\p\in X$. 
\end{lemma}

\begin{theorem}\label{theorem:spl-Bas-var}
Let $L$ be an $S$-submodule of $M$; let $F$ be an $R$-submodule of $\Hom_S(L,N)$; and let $G$ be an $R$-submodule of $\Hom_S(M,N)$.  Suppose that every member of $F$ can be extended to a member of $G$.  Then the following statements hold:
\begin{enumerate}
\item Let $X$ be a subset of $\Supp_R(N)$ that is a basic set for $R$ with $\dim(X)<\infty$.  Let $t$ be a positive integer, and suppose that $\delta_{\p}(F)\geqslant t+\dim(X)$ for every $\p\in X$.  Then there exists $g\in G^{\oplus t}$ such that $g_{\p}$ is split surjective over $S_{\p}$ for every $\p\in X$.
\item Suppose that $Y:=\Max(R)\cap\Supp_R(N)$ is Noetherian with $\dim(Y)<\infty$.  Then the following statements hold:
\begin{enumerate}
\item Let $t$ be a positive integer, and suppose that $\delta_{\m}(F)\geqslant t+\dim(Y)$ for every $\m\in Y$.  Then $\delta(G)\geqslant t$.
\item If $\delta_{\m}(F)=\infty$ for every $\m\in Y$, then $\delta(G)=\infty$.  Hence $\delta(F)=\infty$ if and only if $\delta_{\p}(F)=\infty$ for every $\m\in Y$.
\item Suppose that $\delta_{\n}(F)<\infty$ for some $\n\in Y$.  Then
\[
\delta(G)\geqslant\min\{\delta_{\m}(F):\m\in Y\}-\dim(Y).
\]
\end{enumerate}
\end{enumerate}
\end{theorem}

Let $F$ be a finitely generated $R$-submodule of $\Hom_S(M,N)$, and let $X$ be a subset of $\Supp_R(N)$ that is a basic set for $R$.  The next lemma, which can be compared to Lemma~\ref{lemma:sur-Lambda}, shows that there is a finite subset $\mathit{\Lambda}$ of $X$ that completely determines the function on $X$ taking $\p$ to $\delta_{\p}(F)$.  We omit the proof since it is basically the same as the ones for~\cite[Lemmas~3.6 and~4.2]{DSPY}.

\begin{lemma}\label{lemma:spl-Lambda}
Let $F$ be a finitely generated $R$-submodule of $\Hom_S(M,N)$, and let $X$ be a subset of $\Supp_R(N)$ that is a basic set for $R$.  Then there exists a finite subset $\mathit{\Lambda}$ of $X$ such that, for every $\p\in X-\mathit{\Lambda}$, there exists $\q\in\mathit{\Lambda}$ with the properties that $\q\subsetneq\p$ and $\delta_{\q}(F)=\delta_{\p}(F)$.
\end{lemma}

Using this lemma, we can prove the following analogue of Corollary~\ref{corollary:sur-prelim}.  We omit the proof on account of its similarity with the proof of Corollary~\ref{corollary:sur-prelim}.

\begin{corollary}\label{corollary:spl-prelim}
We make the following improvements to Theorems~\ref{theorem:spl-prelim} and~\ref{theorem:spl-Ser}:
\begin{enumerate}
\item Assume the hypotheses of Part~(2) of Theorem~\ref{theorem:spl-prelim}, and let $\mathit{\Lambda}$ be defined as in Lemma~\ref{lemma:spl-Lambda} with respect to $F$ and $X$.  Then
\[
\delta(G)\geqslant\inf\{\delta_{\p}(F)-\dim_X(\p):\p\in\mathit{\Lambda}\}.
\]
\item Assume the hypotheses of Theorem~\ref{theorem:spl-Ser}, and let $\mathit{\Lambda}$ be defined as in Lemma~\ref{lemma:spl-Lambda} with respect to $F:=\Hom_S(M,N)$ and $X$.  Then
\[
\spl_S(M,N)\geqslant\inf\{\spl_{S_{\p}}(M_{\p},N_{\p})-\dim_X(\p):\p\in\mathit{\Lambda}\}.
\]
\end{enumerate}
\end{corollary}

We now return to the task of reducing the proof of the Splitting Lemma to the study of a finite subset $\mathit{\Lambda}$ of $X$.  To this end, we present the following analogues of Lemmas~\ref{lemma:sur-del-one} and~\ref{lemma:sur-reduction}.  We omit the proofs.

\begin{lemma}\label{lemma:spl-del-one}
Let $n\in\ZZ$ with $n\geqslant 2$; let $\p\in\Supp_R(N)$; let $f:=(f_1,\ldots,f_n)^{\top}\in \Hom_S(M,N^{\oplus n})$; and let $A\in \textnormal{\textbf{GL}}(n,R)$.  Then, with respect to Definition~\ref{definition:sur-F}, we have $\delta_{\p}(F')\geqslant \delta_{\p}(F)-1$.
\end{lemma}

\begin{lemma}\label{lemma:spl-reduction}
Assume the hypotheses of the Splitting Lemma, and define $\mathit{\Lambda}$ as in Lemma~\ref{lemma:spl-Lambda} with respect to $F:=Rf_1+\cdots+Rf_n$ and $X$.  Let $A\in \textnormal{\textbf{GL}}(n,R)$, and suppose that, with respect to Definition~\ref{definition:sur-F}, we have that $f'$ is $(t,X,\mathit{\Lambda})$-split.  Then $f'$ is $(t,X,X)$-split.
\end{lemma}

For the rest of this section, we assume the hypotheses of the Splitting Lemma, and we let $\mathit{\Lambda}$ be defined as in Lemma~\ref{lemma:spl-Lambda} with respect to $F:=Rf_1+\cdots+Rf_n$ and $X$.  Since $t$ and $X$ are understood, we can use the terms \textit{$\p$-split} and \textit{$Y$-split} for any $\p\in X$ and for any $Y\subseteq X$ without the risk of confusion.

Given the next two lemmas, we can find a matrix $V\in\GL(n,R)$ such that the first $n-1$ components of $Vf:=(g_1,\ldots,g_n)^{\top}$ form a map $(g_1,\ldots,g_{n-1})^{\top}$ that is $\m$-split for every $\m\in\mathit{\Lambda}\cap\Max(R)$.  The proofs of the following two lemmas are basically the same as the proofs of Lemmas~\ref{lemma:sur-L} and~\ref{lemma:sur-Max}, and so we omit them.  We make one note, however:  In the proof of Lemma~\ref{lemma:spl-Max} below, we use Lemma~\ref{lemma:sur-Q} by defining 
\[
\mathit{\Lambda}_i:=\{\m\in\mathit{\Lambda}\cap\Max(R):\mathscr{L}_{\m}=i\}
\]
for every $i\in\{1,\ldots,n\}$, where we can find an appropriate choice of $\mathscr{L}_{\m}$ for every $\m\in\mathit{\Lambda}\cap\Max(R)$ by using Lemma~\ref{lemma:spl-L}.

\begin{lemma}\label{lemma:spl-L}
Let $\m\in\mathit{\Lambda}\cap\Max(R)$.  Then there exist elements $r_{\m,1},\ldots,r_{\m,n-1}\in R$ and a number $\mathscr{L}_{\m}\in\{1,\ldots,n\}$ with the following property:  For all $s_1,\ldots,s_n\in R-\m$ and for every $n\times n$ matrix $V$ with entries in $R$ such that 
\[
V
\equiv
\begin{pmatrix}
1 & 0 & \cdots & 0 & r_{\m,1} \\
0 & \ddots & \ddots & \vdots & \vdots \\
\vdots & \ddots & \ddots & 0 & \vdots \\
0 & \cdots & 0 & 1 & r_{\m,n-1} \\
0 & \cdots & \cdots & 0 & 1 \\
\end{pmatrix}
P_{\mathscr{L}_{\m}}
\begin{pmatrix}
s_1 & 0 & \cdots & 0 & 0 \\
0 & \ddots & \ddots & \vdots & \vdots \\
\vdots & \ddots & \ddots & 0 & \vdots \\
0 & \cdots & 0 & s_{n-1} & 0 \\
0 & \cdots & \cdots & 0 & s_n \\
\end{pmatrix}
\textnormal{ (mod }\m\textnormal{)},
\]
the first $n-1$ components of $Vf:=(g_1,\ldots,g_n)^{\top}$ form a map $(g_1,\ldots,g_{n-1})^{\top}$ that is $\m$-split.
\end{lemma}

\begin{lemma}\label{lemma:spl-Max}
There exists a matrix $V\in \textnormal{\textbf{GL}}(n,R)$ such that the first $n-1$ components of $Vf:=(g_1,\ldots,g_n)^{\top}$ form a map $(g_1,\ldots,g_{n-1})^{\top}$ that is $\m$-split for every $\m\in \mathit{\Lambda}\cap\Max(R)$.
\end{lemma}

Lastly, we must find a matrix $V\in \GL(n,R)$ such that the first $n-1$ components of $Vf:=(g_1,\ldots,g_n)^{\top}$ form a map $g:=(g_1,\ldots,g_{n-1})^{\top}$ that is $\mathit{\Lambda}$-split.  Lemma~\ref{lemma:spl-reduction} will then tell us that $g$ is $X$-split and, hence, that we have proved the Splitting Lemma.

\begin{proof}[Proof of the Splitting Lemma]
The beginning of the proof is basically the same as the proof of the Surjective Lemma up to, and including, the point where we reduce to the case in which a matrix $B$ has a certain desirable form with $(Bf^*)_{\q}$ surjective.  Of course, here, we need $(Bf^*)_{\q}$ to be not only surjective but also split surjective over $S_{\q}$.

Let $d:=\delta_{\q}(F)$.  Since $(Bf^*)_{\q}$ is split surjective over $S_{\q}$, there exists an $S$-submodule $L$ of $M$ such that the restriction of $(Bf^*)_{\q}$ to $L_{\q}$ is an isomorphism.  We may assume, then, without loss of generality, that $M_{\q}=N_{\q}^{\oplus d}$.

It remains to find $r_1,\ldots,r_{n-1}\in J$ such that, if
\[
U
:=\begin{pmatrix}
1 & 0 & \cdots & 0 & r_1 \\
0 & \ddots & \ddots & \vdots & \vdots \\
\vdots & \ddots & \ddots & 0 & \vdots \\
0 & \cdots & 0 & 1 & r_{n-1} \\
0 & \cdots & \cdots & 0 & 1 \\
\end{pmatrix}
\]
and if $Uf^*:=(g_1,\ldots,g_n)^{\top}$, then $G:=Rg_1+\cdots+Rg_{n-1}$ satisfies $\delta_{\q}(G)=d$.  On the other hand, since we have reduced to the case in which $M_{\q}=N_{\q}^{\oplus d}$, it suffices to verify that $\partial_{\q}(G)=d$, for this will imply that $\delta_{\q}(G)=d$.  Thus, from here, we may proceed once more as in the proof of the Surjective Lemma.
\end{proof}

As with the Surjective Lemma, there is a special case that admits a stronger version of the Splitting Lemma.  The reasoning is virtually identical to the discussion preceding Corollary~\ref{corollary:sur-res}, so we simply state the result here.

\begin{corollary}\label{corollary:spl-res}
Assume the hypotheses of the Splitting Lemma.  Define $\mathit{\Lambda}$ as in Lemma~\ref{lemma:spl-Lambda} with respect to $F:=Rf_1+\cdots+Rf_n$ and $X$.  Suppose that, for every $\m\in\mathit{\Lambda}\cap\Max(R)$, one of the following conditions holds:
\begin{enumerate}
\item $|R/\m|\geqslant 2+\mu_{R_{\m}}(N_{\m})$.
\item $\mu_{R_{\m}}(N_{\m})=1$.
\end{enumerate}
(For example, we may suppose that every residue field of $R$ is infinite or that $N$ is a locally cyclic $R$-module.)  Then there exist $r_1,\ldots,r_{n-1}\in R$ such that $(f_1+r_1f_n,\ldots,f_{n-1}+r_{n-1}f_n)^{\top}$ is $(t,X,X)$-split.
\end{corollary}

It would now be possible to state improved versions of Theorems~\ref{theorem:spl-Bas},~\ref{theorem:spl-Ser},~\ref{theorem:spl-prelim}, and~\ref{theorem:spl-Bas-var} in the special case acknowledged by the previous corollary, but we omit the details.

In the next section, we consider another special case in which we can improve upon our previous results.

\section{Finitely Generated Modules over Dedekind Domains}\label{sec:char}

In this section, we characterize global surjective and splitting capacities of finitely generated modules over Dedekind domains.  We define a \textit{Dedekind domain} to be an integral domain in which every ideal is projective.  Hence we consider a field to be a Dedekind domain.  Since every ideal of a Dedekind domain is projective, every ideal is finitely generated~\cite[pages~760--762]{DF}, and so a Dedekind domain is always Noetherian.

A \textit{fractional ideal} of a Dedekind domain $R$ is an $R$-submodule $I$ of the fraction field of $R$ such that there exists a nonzero $a\in R$ with $aI$ an ideal of $R$.  Hence every fractional ideal of $R$ is isomorphic to an ideal of $R$.  We define an equivalence relation $\sim$ on the set $\mathscr{F}$ of all nonzero fractional ideals of $R$ by letting $I\sim J$ if and only if there exist nonzero $a,b\in R$ such that $aI=bJ$.  The set of all equivalence classes of $\mathscr{F}$ with respect to $\sim$ forms an abelian group under multiplication with $[I][J]:=[IJ]$ for all nonzero ideals $I$ and $J$ of $R$.  This group, denoted $\Pic(R)$, is called the \textit{Picard group of $R$} or the \textit{class group of $R$}, and its identity is $[R]$, the class of all principal fractional ideals of $R$.  See~\cite[pages~457--460]{AB};~\cite[pages~760--762]{DF}; or~\cite[pages~253--258]{Eis} for more details.

For every module $M$ over a Dedekind domain $R$, we define $\Tor_R(M)$ to be the torsion $R$-submodule of $M$.  We omit the subscript $R$ if the underlying ring is understood.  Indeed, in this section, we will often omit subscripts referring to rings in notation such as $\Ass_R(N)$, $\Supp_R(N)$, $\sur_R(M,N)$, and $\spl_{R_{\p}}(M_{\p},N_{\p})$ since, in many cases, the underlying ring will be understood.  In particular, we do not consider arbitrary module-finite algebras over Dedekind domains here.

We now recall the structure theorem for finitely generated modules over Dedekind domains.  The various parts of this theorem can be found in~\cite[pages~771 and~774]{DF}.  Alternatively, Parts~(1),~(3), and~(4) can be found in~\cite[page~463]{AB} or~\cite[pages~484--485]{Eis}, and Part~(2) can be deduced by applying the Chinese Remainder Theorem to results from~\cite[page~458]{AB} or~\cite[page~258]{Eis}.

\begin{theorem}\label{theorem:str}
Let $M$ be a finitely generated module over a Dedekind domain $R$.  Then the following statements hold:
\begin{enumerate}
\item $M\cong \Tor(M) \oplus M/\Tor(M)$.
\item $\Tor(M)$ is a direct sum of $R$-modules, each of which has the form $R/\m^i$ for some $\m\in\Spec(R)-\{0\}$ and some positive integer $i$.  This decomposition is unique up to a permutation of factors.
\item There is an alternative decomposition of $\Tor(M)$ as $(R/I_1) \oplus\cdots\oplus (R/I_u)$ for some nonzero proper ideals $I_1,\ldots,I_u$ of $R$ such that $I_1\subseteq\cdots\subseteq I_u$.  This decomposition is unique.
\item Suppose that $M\neq\Tor(M)$.  Then there is a nonzero ideal $I$ of $R$ such that $M/\Tor(M)$ $\cong R^{\oplus (r-1)}\oplus I$, where $r:=\rank(M)$.  The ideal $I$ is unique up to isomorphism. 
\end{enumerate}
\end{theorem}

Part~(2) of the preceding theorem gives the \textit{elementary divisor decomposition} of $\Tor(M)$, and Part~(3) gives the \textit{invariant factor decomposition} of $\Tor(M)$.  If $M\neq\Tor(M)$, then the class $[I]$ of the ideal $I$ from Part~(4) is called the \textit{Steinitz class of $M$} and is denoted by $[M]$.  See~\cite[page~773]{DF}.

We collect a few more properties of Dedekind domains in the following lemma.

\begin{lemma}\label{lemma:RIJ}
Let $R$ be a Dedekind domain.  Then $R$ satisfies the following properties:
\begin{enumerate}
\item Let $I$ and $J$ be nonzero ideals of $R$.  Then $I\oplus J \cong R \oplus IJ$.
\item Let $I$ and $J$ be nonzero ideals of $R$.  Then there is a nonzero ideal $K$ of $R$ such that $I\cong KJ$.
\item Let $I$ be a nonzero ideal of $R$, and let $M$ be a cyclic torsion $R$-module.  Then there is a surjective $R$-linear map from $I$ to $M$.
\item Let $M$ be a finitely generated torsion $R$-module, and let $u$ be a positive integer.  Suppose that $\mu_{R_{\m}}(M_{\m})\leqslant u$ for every $\m\in\Ass(M)$.  Then $\mu_R(M)\leqslant u$.
\end{enumerate}
\end{lemma}

\begin{proof}
(1)  See~\cite[pages~461--462]{AB};~\cite[page~769]{DF}; or~\cite[page~484]{Eis}.

(2)  Since $\Pic(R)$ is a group, there is an ideal $K$ of $R$ such that $[I]=[K][J]=[KJ]$.  Hence $I\cong KJ$. 

(3)  The result is obvious if $M=0$, so suppose that $M\neq 0$.  Then there is a nonzero proper ideal $J$ of $R$ such that $M\cong R/J$.  From~\cite[page~765]{DF}, we learn that $I$ and $J$ can be written as $I=\m_1^{v_1}\cdots\m_u^{v_u}$ and $J=\m_1^{w_1}\cdots\m_u^{w_u}$, where $\m_1,\ldots,\m_u$ are distinct nonzero prime ideals of $R$ and where $v_1,\ldots,v_u,w_1,\ldots,w_u$ are nonnegative integers.  Now let $K:=\m_1^{v_1+w_1}\cdots\m_u^{v_u+w_u}$.  By results from~\cite[page~768]{DF}, we see that $\m_i^{v_i}/\m_i^{v_i+w_i}\cong R/\m_i^{w_i}$ for every $i\in\{1,\ldots,u\}$ and, hence, that $I/K\cong R/J$.

(4)  This follows from the invariant factor decomposition of $\Tor(M)$ in Theorem~\ref{theorem:str}.
\end{proof}

Here is our first main result on Dedekind domains:

\begin{proposition}\label{proposition:sur-Ded}
Let $M$ and $N$ be finitely generated modules over a Dedekind domain $R$; let $X:=\Ass(N)-\{0\}$; and let $t$ be a positive integer.  Then $\sur(M,N)\geqslant t$ if and only if $\sur(M_{\m},N_{\m})\geqslant t$ for every $\m\in X$ and one of the following conditions holds:
\begin{enumerate}
\item $\rank(N)=0$.
\item $\rank(M)\geqslant 1+t\cdot\rank(N)$.
\item $\rank(M)=t\cdot\rank(N)\geqslant t$, and $[M]=[N]^t$.
\end{enumerate}
Moreover, if $\sur(M,N)\geqslant t$ and (3) holds, then we have that $\sur(\Tor(M),\Tor(N))\geqslant t$ and $\sur(M,N)=~t$.
\end{proposition}

\begin{proof}
Let $r:=\rank(M)$, and let $s:=\rank(N)$.

Suppose first that $\sur(M,N)\geqslant t$.  Certainly $\sur(M_{\m},N_{\m})\geqslant t$ for every $\m\in X$.  Suppose that neither (1) nor (2) holds.  Then $r=st\geqslant t$, and so $\sur(M,N)=t$.  Let $I,J$ be ideals of $R$ that represent $[M],[N]$, respectively.  Then $N^{\oplus t}/\Tor(N^{\oplus t})\cong R^{\oplus (st-1)}\oplus J^t$ by Part~(1) of Lemma~\ref{lemma:RIJ}.  Let $e$ be a surjective $R$-linear map from $M$ to $N^{\oplus t}$.  Note that $\Hom_R(\Tor(M),R^{\oplus (st-1)} \oplus J^t)=0$.  Hence there exist $R$-linear maps
\[
\begin{array}{rclcrl}
f & : & \Tor(M) & \rightarrow & \Tor(N)^{\oplus t} & , \\
g & : & R^{\oplus (st-1)} \oplus I & \rightarrow & \Tor(N)^{\oplus t} & , \\
h & : & R^{\oplus (st-1)} \oplus I & \rightarrow & R^{\oplus (st-1)} \oplus J^t & \\
\end{array}
\]
such that the following matrix represents $e$:
\[
\bordermatrix{
& \Tor(M) & R^{\oplus (st-1)} \oplus I \cr
\Tor(N)^{\oplus t} & f & g  \cr
R^{\oplus (st-1)} \oplus J^t & 0 & h  \cr
}.
\]
Clearly, $h$ is surjective.  Since $R^{\oplus (st-1)} \oplus I$ is torsion-free, $\ker h$ is torsion-free.  Since $\rank(\ker h)=\rank(M)-\rank(N^{\oplus t})=0$, we see that $\ker h=0$.  Hence $h$ is an isomorphism, and so Theorem~\ref{theorem:str} tells us that $I\cong J^t$.  Thus $[M]=[N]^t$, proving (3).  As a result, the Snake Lemma tells us that $f$ is surjective.  Hence $\sur(\Tor(M),\Tor(N))\geqslant t$.

Next suppose that $\sur(M_{\m},N_{\m})\geqslant t$ for every $\m\in X$.  If (1) holds, then clearly it is the case that $\sur(M,N)\geqslant t$. 

Suppose then that (1) does not hold but that (2) does hold.  Let $I,J$ be ideals of $R$ that represent $[M],[N]$, respectively.  By Part~(2) of Lemma~\ref{lemma:RIJ}, there exists a nonzero ideal $K$ of $R$ such that $I\cong KJ^t$.  By Part~(1) of Lemma~\ref{lemma:RIJ}, we may then write $M/\Tor(M)$ as
\[
R^{\oplus (r-1)} \oplus I \cong R^{\oplus (r-st-1)}\oplus K \oplus R^{(st-1)} \oplus J^t.
\]
If $X=\varnothing$, then immediately we see that $\sur(M,N)\geqslant t$ since $N^{\oplus t}/\Tor(N^{\oplus t})\cong R^{\oplus (st-1)}\oplus J^t$ by Part~(1) of Lemma~\ref{lemma:RIJ}.  Suppose then that $X\neq\varnothing$.  Let $\m\in X$, and let $e(\m)$ be a surjective $R$-linear map from $M_{\m}$ to $N_{\m}^{\oplus t}$.  Note that $\Hom_R(\Tor(M_{\m}),R_{\m}^{\oplus st})=0$. Hence there exist $R$-linear maps
\[
\begin{array}{rclcrl}
f(\m) & : & \Tor(M_{\m}) & \rightarrow & \Tor(N_{\m})^{\oplus t} & , \\
g(\m) & : & R_{\m}^{\oplus r} & \rightarrow & \Tor(N_{\m})^{\oplus t} & , \\
h(\m) & : & R_{\m}^{\oplus r} & \rightarrow & R_{\m}^{\oplus st} & \\
\end{array}
\]
such that the following matrix represents $e(\m)$:
\[
\bordermatrix{
& \Tor(M_{\m}) & R_{\m}^{\oplus r} \cr
\Tor(N_{\m})^{\oplus t} & f(\m) & g(\m) \cr
R_{\m}^{\oplus st} & 0 & h(\m)  \cr
}.
\]
Clearly, $h(\m)$ is surjective.  Since $R_{\m}^{\oplus st}$ is free over $R_{\m}$, we see that $\ker h(\m)\cong R_{\m}^{\oplus (r-st)}$.  As a result, the Snake Lemma tells us that $\mu_{R_{\m}}(\coker f(\m))\leqslant r-st$.  Lift a generating set of $\coker f(\m)$ to a subset $\mathscr{C}(\m)$ of $\Tor(N_{\m})^{\oplus t}\subseteq\Tor(N)^{\oplus t}$.  Let $\m_1,\ldots,\m_u$ be the distinct members of $X$, and let $C$ be the $R$-module generated by $\mathscr{C}(\m_1)\cup\cdots\cup \mathscr{C}(\m_u)$.  Then, by Part~(4) of Lemma~\ref{lemma:RIJ}, we see that $\mu_R(C)\leqslant r-st$.  Hence, by Part~(3) of Lemma~\ref{lemma:RIJ}, there is a surjective $R$-linear map from $R^{\oplus (r-st-1)}\oplus K$ to $C$.  As a result, there is a surjective $R$-linear map from $\Tor(M)\oplus R^{\oplus (r-st-1)}\oplus K$ to $\im f(\m_1) + \cdots + \im f(\m_u) + C = \Tor(N)^{\oplus t}$.  Altogether, then, we find that $\sur(M,N)\geqslant t$.

Finally, suppose that (3) holds.  If $X=\varnothing$, then immediately we see that $\sur(M,N)=t$.  Suppose then that $X\neq\varnothing$.  Let $\m\in X$, and let $e(\m)$ be a surjective $R$-linear map from $M_{\m}$ to $N_{\m}^{\oplus t}$.  Note that $\Hom_R(\Tor(M_{\m}),R_{\m}^{\oplus st})=0$. Hence there exist $R$-linear maps
\[
\begin{array}{rclcrl}
f(\m) & : & \Tor(M_{\m}) & \rightarrow & \Tor(N_{\m})^{\oplus t} & , \\
g(\m) & : & R_{\m}^{\oplus st} & \rightarrow & \Tor(N_{\m})^{\oplus t} & , \\
h(\m) & : & R_{\m}^{\oplus st} & \rightarrow & R_{\m}^{\oplus st} & \\
\end{array}
\]
such that the following matrix represents $e(\m)$:
\[
\bordermatrix{
& \Tor(M_{\m}) & R_{\m}^{\oplus st} \cr
\Tor(N_{\m})^{\oplus t} & f(\m) & g(\m) \cr
R_{\m}^{\oplus st} & 0 & h(\m)  \cr
}.
\]
Clearly, $h(\m)$ is surjective.  Since $R_{\m}^{\oplus st}$ is finitely generated over $R_{\m}$, we see that $h(\m)$ is an isomorphism.  As a result, the Snake Lemma tells us that $f(\m)$ is surjective.  Hence $\sur(\Tor(M_{\m}),\Tor(N_{\m}))\geqslant t$, and so $\sur(\Tor(M),\Tor(N))\geqslant t$.  Moreover, since $\rank(M)=\rank(N^{\oplus t})$ and since $[M]=[N]^t$, Theorem~\ref{theorem:str} and Part~(1) of Lemma~\ref{lemma:RIJ} tell us that $M/\Tor(M)\cong N^{\oplus t}/\Tor(N^{\oplus t})$.  Altogether then, $\sur(M,N)= t$.
\end{proof}

Using the previous proposition, we can characterize the global surjective capacity of a finitely generated module $M$ with a respect to a finitely generated module $N$ over a Dedekind domain $R$:  To see this, let $u$ be a positive integer.  Then $\sur(M,N)=u$ if and only if $\sur(M,N)\geqslant u$ and $\sur(M,N)<u+1$.  We can apply the previous proposition with $t=u$ to characterize the statement that $\sur(M,N)\geqslant u$, and we can apply the previous proposition with $t=u+1$ to characterize the statement that $\sur(M,N)<u+1$.  We also have the following corollary:

\begin{corollary}\label{corollary:sur-Ded}
Let $M$ and $N$ be finitely generated modules over a Dedekind domain $R$, and let $X:=\Ass(N)-\{0\}$.  Then $\sur(M,N)=0$ if and only if one of the following conditions holds:
\begin{enumerate}
\item $\sur(M_{\m},N_{\m})=0$ for some $\m\in X$.
\item $\rank(N)\geqslant 1+\rank(M)$.
\item $\rank(M)=\rank(N)\geqslant 1$, and $[M]\neq [N]$.
\end{enumerate}
\end{corollary}

Now recall the example given in the introduction to this paper:  If $R$ is a Dedekind domain with a nonprincipal ideal $I$, then $\sur(R_{\p},I_{\p})=\sur(I_{\p},R_{\p})=1$ for every $\p\in\Spec(R)$, and yet $\sur(R,I)=\sur(I,R)=0$.  Corollary~\ref{corollary:sur-Ded} extends this example to a complete characterization of the condition that $\sur(M,N)=0$ when $M$ and $N$ are finitely generated modules over a Dedekind domain $R$.

We can give a result analogous to Proposition~\ref{proposition:sur-Ded} for splitting capacities:

\begin{proposition}\label{proposition:spl-Ded}
Let $M$ and $N$ be finitely generated modules over a Dedekind domain $R$; let $X:=\Ass(N)-\{0\}$; and let $t$ be a positive integer.  Then $\spl(M,N)\geqslant t$ if and only if $\spl(M_{\m},N_{\m})\geqslant t$ for every $\m\in X$ and one of the following conditions holds:
\begin{enumerate}
\item $\rank(N)=0$.
\item $\rank(M)\geqslant 1+t\cdot\rank(N)$.
\item $\rank(M)=t\cdot\rank(N)\geqslant t$, and $[M]=[N]^t$.
\end{enumerate}
Moreover, if $\spl(M,N)\geqslant t$ and (3) holds, then $\spl(\Tor(M),\Tor(N))\geqslant t$, and $\spl(M,N)=t$.
\end{proposition}

\begin{proof}
Let $r:=\rank(M)$, and let $s:=\rank(N)$.

Suppose first that $\spl(M,N)\geqslant t$.  Certainly $\spl(M_{\m},N_{\m})\geqslant t$ for every $\m\in X$.  Suppose that neither (1) nor (2) holds.  Then $r=st\geqslant t$, and so $\spl(M,N)=t$.  Let $I,J$ be ideals of $R$ that represent $[M],[N]$, respectively.  Then, by Theorem~\ref{theorem:str} and Part~(1) of Lemma~\ref{lemma:RIJ}, we see that $I\cong J^t$, and so $[M]=[N]^t$, proving (3).  Theorem~\ref{theorem:str} also implies that $\spl(\Tor(M),\Tor(N))\geqslant t$.

Next suppose that $\spl(M_{\m},N_{\m})\geqslant t$ for every $\m\in X$.  If (1) holds, then it must be the case that $\spl(M,N)\geqslant t$. 

Suppose then that (1) does not hold but that (2) does hold.  Let $I,J$ be ideals of $R$ that represent $[M],[N]$, respectively.  By Part~(2) of Lemma~\ref{lemma:RIJ}, there exists a nonzero ideal $K$ of $R$ such that $I\cong KJ^t$.  By Part~(1) of Lemma~\ref{lemma:RIJ}, we may write  $M/\Tor(M)$ as
\[
R^{\oplus (r-1)} \oplus I \cong R^{\oplus (r-st-1)}\oplus K \oplus R^{(st-1)} \oplus J^t.
\]
Now Theorem~\ref{theorem:str} implies that $\spl(M,N)\geqslant t$ since $N^{\oplus t}/\Tor(N^{\oplus t})\cong R^{\oplus (st-1)}\oplus J^t$ by Part~(1) of Lemma~\ref{lemma:RIJ}.

Finally, suppose that (3) holds.  Then Theorem~\ref{theorem:str} and Part~(1) of Lemma~\ref{lemma:RIJ} immediately imply that $\spl(\Tor(M),\Tor(N))\geqslant t$ and that $\spl(M,N)=t$.\end{proof}

In light of the previous proposition, we could now characterize global splitting capacities of finitely generated modules over a Dedekind domain in a manner similar to the case for global surjective capacities.  We omit the details.

Given the wealth of results that we have obtained on surjective and splitting capacities, we propose the following injective analogue of these concepts:

\begin{definition}\label{definition:inj}
Let $R$ be a commutative ring, $S$ an $R$-algebra, and $M$ and $N$ right $S$-modules.  We let $\inj_S(M,N)$ denote the supremum of the nonnegative integers $t$ such that there exists an injective $S$-linear map from $N^{\oplus t}$ to $M$ (the order of the modules here is correct), and we refer to $\inj_S(M,N)$ as the \textit{global injective capacity of $M$ with respect to $N$ over $S$}.

Let $\p\in\Spec(R)$.  We refer to $\inj_{S_{\p}}(M_{\p},N_{\p})$ as the \textit{local injective capacity of $M$ with respect to $N$ over $S$ at $\p$}.
\end{definition}

As with surjective and splitting capacities, we can always get an upper bound on a global injective capacity in terms of local injective capacities:
\[
\inj_S(M,N)\leqslant\inf\{\inj_{S_{\m}}(M_{\m},N_{\m}):\m\in\Max(R)\cap\Supp_R(N)\}.
\]
Can we also get a lower bound in terms of local injective capacities?  Can we characterize a given global injective capacity using local injective capacities?  We have not been able to find any answers to these questions in the existing literature, and we cannot provide results that are as far-reaching as our results regarding surjective and splitting capacities.  However, in the case of finitely generated modules over a Dedekind domain, we can provide an answer.  As with surjective and splitting capacities, we can write $\inj(M,N)$ and $\inj(M_{\p},N_{\p})$ for every $\p\in\Spec(R)$ without causing confusion concerning the underlying ring.

\begin{proposition}\label{proposition:inj-Ded}
Let $M$ and $N$ be finitely generated modules over a Dedekind domain $R$, and let $X:=\Ass(N)$.  Then
\[
\inj(M,N)=\inf\{\inj(M_{\p},N_{\p}):\p\in X\}.
\]
\end{proposition}

\begin{proof}
Let $r:=\rank(M)$; let $s:=\rank(N)$; and let
\[
t:=\inf\{\inj(M_{\p},N_{\p}):\p\in X\}.
\]
It is clear that $\inj(M,N)\leqslant t$, and so it remains to prove that $\inj(M,N)\geqslant t$.  If $t=0$, then there is nothing to prove.  If $t=\infty$, then $X=\varnothing$, and so $\inj(M,N)=t$.  Suppose then that $t$ is a positive integer so that $X\neq\varnothing$.  If $s=0$, then Theorem~\ref{theorem:str} immediately tells us that $\inj(M,N)\geqslant t$.  Suppose then that $s\geqslant 1$ so that $r\geqslant st\geqslant 1$.  Let $I,J$ be ideals of $R$ that represent $[M],[N]$, respectively.  Let $a\in I-\{0\}$, and let $f:J^t\rightarrow I$ be the $R$-linear map defined by letting $f(x)=ax$ for every $x\in J^t$. Then $f$ is injective, and so $\inj(I,J^t)=1$.  Thus Theorem~\ref{theorem:str} and Part~(1) of Lemma~\ref{lemma:RIJ} tell us that $\inj(M,N)\geqslant t$.
\end{proof}

To close, we mention a few more cases in which we can characterize global surjective, splitting, and injective capacities.  We have already covered some of these cases.  For example, a characterization of an infinite global surjective capacity can be found in Part~(2) of Theorem~\ref{theorem:sur-Bas}, and Part~(2) of Theorem~\ref{theorem:spl-Bas} offers a characterization of an infinite global splitting capacity.  Under the hypotheses of Part~(3) of Theorem~\ref{theorem:sur-Bas}, if $\dim(Y)=0$, then
\[
\sur_S(M,N)=\min\{\sur_{S_{\m}}(M_{\m},N_{\m}):\m\in Y\}
\]
since we always have 
\[
\sur_S(M,N)\leqslant\inf\{\sur_{S_{\m}}(M_{\m},N_{\m}):\m\in Y\}.
\]
We can draw an analogous conclusion from Part~(3) of Theorem~\ref{theorem:spl-Bas} when $\dim(Y)=0$:  In this case,
\[
\spl_S(M,N)=\min\{\spl_{S_{\m}}(M_{\m},N_{\m}):\m\in Y\}.
\]
In particular, if our underlying commutative ring is \textit{quasisemilocal} (the ring has only finitely many maximal ideals), then global surjective capacities are completely determined by local surjective capacities, and we can say the same for splitting capacities.

We now observe that, whenever we have a result concerning surjective, splitting, or injective capacities over finitely many commutative rings $R_1,\ldots,R_u$, we have an analogous result for the direct product $R_1\times\cdots\times R_u$ of these rings.  For example, since a \textit{commutative Noetherian hereditary ring} is a direct product of finitely many Dedekind domains, we can characterize global surjective, splitting, and injective capacities of finitely generated modules over any commutative Noetherian hereditary ring, given Propositions~\ref{proposition:sur-Ded},~\ref{proposition:spl-Ded}, and~\ref{proposition:inj-Ded}.  In light of our additional results on global surjective and splitting capacities over commutative quasisemilocal rings, we can characterize global surjective and splitting capacities over any direct product of finitely many commutative quasisemilocal rings and Dedekind domains.  One historically significant example of such a direct product is given by Hungerford's Theorem~\cite[Theorem~1]{Hun}:  A \textit{commutative principal ideal ring} (that is, a commutative ring in which every ideal is principal) is a direct product of finitely many quotients of principal ideal domains.  Direct products thus provide a way to extend some of the results of this paper to larger classes of rings.

\bibliographystyle{amsplain}
\bibliography{COHM-bib}

\end{document}